\documentclass[11pt,reqno,a4paper]{amsart}
\usepackage{amsmath,amssymb,verbatim,color,enumerate,graphicx,caption}
\usepackage{amsthm,amsfonts}%
\usepackage{mathrsfs}%
\usepackage{resizegather}
\usepackage{color,comment}
\usepackage[normalem]{ulem}

\usepackage{mathtools}
\mathtoolsset{showonlyrefs}

\newtheorem{thm}{Theorem}[section]
\newtheorem{cor}[thm]{Corollary}
\newtheorem{lem}[thm]{Lemma}

\newtheorem*{conj*}{Conjecture}
\newtheorem{remark}[thm]{Remark}
\newtheorem{thmx}{Theorem}
\newtheorem{corx}[thmx]{Corollary}

\numberwithin{equation}{section}

\newcommand{\ab}{\varrho}
\newcommand{\bo}{{\rm O}}
\newcommand{\diri}[2]{\lambda_{#1}^{\left(#2\right)}}
\newcommand{\navi}[2]{\sigma_{#1}^{\left(#2\right)}}
\newcommand{\ds}{\displaystyle}
\newcommand{\dint}{\ds\int}
\newcommand{\dsum}{\ds\sum}
\newcommand{\eqskip}{ \vspace*{2mm}\\ }
\newcommand{\fr}[2]{\frac{\ds #1}{\ds #2}}
\newcommand{\dintd}[2]{{\ds \int_{\ds #1}^{\ds #2}}}
\newcommand{\dprod}{\ds\prod}
\newcommand{\so}{{\rm o}}

\begin{document}

\title[Inequalities and asymptotics for polyharmonic eigenvalues]{Sharp inequalities and asymptotics for polyharmonic eigenvalues}
\author{D.\ Buoso}
\author{P.\ Freitas}

\address{Dipartimento per lo Sviluppo Sostenibile e la Transizione Ecologica, Universit\`a degli Studi del Piemonte Orientale ``A.\ Avogadro'', piazza Sant'Eusebio 5,
13100 Vercelli, Italy}
\email{davide.buoso@uniupo.it}

\address{Departamento de Matem\'atica, Instituto Superior T\'ecnico, Universidade de Lisboa, Av. Rovisco Pais 1,
P-1049-001 Lisboa, Portugal}
\email{pedrodefreitas@tecnico.ulisboa.pt}


\begin{abstract}
 We study eigenvalues of general scalar Dirichlet polyharmonic problems in domains in $\mathbb R^{d}$. We first
 prove a number of inequalities satisfied by the eigenvalues on general domains, depending on the relations
 between the orders of the operators involved. We then obtain several estimates for these eigenvalues, yielding their growth as a function of these orders. For the problem in the ball we derive the general form of eigenfunctions together with the equations satisfied by the corresponding eigenvalues, and obtain several bounds for the first eigenvalue. In the case of the polyharmonic operator of order $2m$ we derive precise bounds yielding the first two terms in the asymptotic expansion for the first normalised eigenvalue as $m$ grows to infinity.
These results allow us to obtain the order of growth for the $k^{\rm th}$ polyharmonic
 eigenvalue on general domains.
\end{abstract}

\maketitle

\tableofcontents

\section{Introduction}

Let $\Omega$ be a domain (i.e., a connected open set) of finite Lebesgue measure in $\mathbb R^d$, $d\ge1$. Given positive integers
$t\leq m$ we consider the following class of eigenvalue
problems involving polyharmonic operators 
\begin{equation}
\label{dirichletbuckling}
\left\{\begin{array}{ll}
(-\Delta)^mu=\diri{}{m,t} (-\Delta)^{m-t}u, & {\rm \ in\ }\Omega,\eqskip
u=\fr{\partial u}{\partial \nu}=\dots=\fr{\partial^{m-1} u}{\partial \nu^{m-1}}=0, & {\rm \ on\ }\partial \Omega.
\end{array}\right.
\end{equation}
Here $\Delta$ denotes the usual Laplacian, and $\nu$ is the outer unit normal to the domain $\Omega$; the corresponding set of boundary conditions is usually
referred to as Dirichlet boundary conditions in the literature.
Higher-order elliptic problems of this and similar type have been intensively studied in the literature for more than a century, not only due to their relations to applications,
but also because they are an interesting mathematical object in their own right~\cite{almansi,cqw,freilipo,ggs,grs,jlwx2,jlwx,kkm,leyl,puccserr,safvass,yoyo}.

For specific choices of the parameters $m$ and $t$ the above general problem
specialises into well-known problems such as the Dirichlet Laplacian ($m=t=1$), the Dirichlet bilaplacian ($m=t=2$), the buckling problem ($m=2$ and $t=1$),
and for any pair of integers $m$ and $t$ with $m=t$ we have the Dirichlet problem for the general polyharmonic operator, namely,
\begin{equation}
\label{dirichlet}
\left\{\begin{array}{ll}
(-\Delta)^mu=\diri{}{m,m} u, & {\rm \ in\ }\Omega,\eqskip
u=\fr{\partial u}{\partial \nu}=\dots=\fr{\partial^{m-1} u}{\partial \nu^{m-1}}=0, & {\rm \ on\ }\partial \Omega.
\end{array}\right.
\end{equation}
In this last instance, and to simplify notation, we will refer to general polyharmonic eigenvalues by $\diri{}{m,m}=\diri{}{m}$.

It is known that under the conditions considered the spectrum of the above problem is discrete and constitutes an increasing sequence of the form~\cite{bula}
\[
 0< \diri{1}{m,t}(\Omega)\leq \diri{2}{m,t}(\Omega) \leq \diri{3}{m,t}(\Omega) \leq \cdots,
\]
with $\diri{k}{m,t}(\Omega) \to +\infty$ as $k\to\infty$. The corresponding asymptotic behaviour in $k$ is given by the following Weyl law~\cite{birsol1,birsol2}
(see also \cite{roz1,roz2,safvass,vass})
\begin{equation}
\label{weyl1}
\diri{k}{m,t}(\Omega)=(2\pi)^{2t}\left(\fr{k}{\omega_d|\Omega|}\right)^{\frac{2t}{d}}+ \so\left(k^{\frac {2t}{d}}\right),
\end{equation}
for any (smooth enough) bounded domain $\Omega$, where $\omega_d$ is the measure of the $d$-dimensional unit ball and $|\Omega|$ denotes
the ($d$-dimensional) measure of $\Omega$. As pointed out above, the case $m=t=1$ corresponds to the usual Dirichlet Laplacian, and the first term in the above asymptotics
is just the first term of the Weyl law for the Laplacian to the power $t$. 
We thus see that the parameter $t$, which encodes the difference in the differential orders
of the two operators involved in the eigenvalue problem~\eqref{dirichletbuckling}, affects the rate of growth in the order of the eigenvalues, while the paramenter $m$
which corresponds to half of the actual differential order of the equation does not appear in this first term. These observations raise two issues. The first
is that a natural normalisation for the eigenvalues of problem~\eqref{dirichletbuckling} is to consider their $(2t)^{\rm th}$ root. In this way, the first term in the Weyl
law is always the same (for a given fixed dimension), namely,~\eqref{weyl1} may now be written as
\begin{equation}
\label{weylmt}
\left(\diri{k}{m,t}(\Omega)\right)^{\frac{1}{2t}}=2\pi\left(\fr{k}{\omega_d|\Omega|}\right)^{\frac{1}{d}}+ \so\left(k^{\frac{1}{d}}\right),
\end{equation}
and, maybe more relevant, it thus produces quantities that are comparable -- this normalisation is widely used in the literature, as for instance in~\cite{safvass}, where the 
eigenvalue problem for a positive definite elliptic self-adjoint operator $A$ of order $2m$ is written as $Au = \gamma^{2m} u$, or in~\cite{erve}, where the relevant parameter
appearing in a control problem is precisely the first eigenvalue of~\eqref{dirichlet} in one dimension, normalised in this way -- see Remark~\ref{remerve} below
for more details.

The second issue is that this perspective, at least at this level and to some extent, is not taking into consideration the actual effect of the differential order of
the equation on the eigenvalues. To be more specific, by the precise nature of Weyl's law, the above considerations relate to the asymptotic behaviour in the order
of the eigenvalue, leaving out the dependence of, for instance, the first eigenvalue on the differential order of the operator. Apart from being a natural question
of theoretical interest, this behaviour is also of actual relevance in applications and, to the best of our knowledge, is not known even in the one-dimensional case.
In fact, most of the literature on the subject has been devoted to the study of sums of eigenvalues and inequalities of universal type, whose asymptotic sharpness
is more closely related to the dependence of the eigenvalues on their order rather than on the actual differential order of the operators involved -- see, for
instance,~\cite{cqw,jlwx2,jlwx,yoyo}. For the former type of problems, the dependence remains, to leading order, independent of the boundary conditions considered.
However, this will no longer be the case for the growth of a fixed eigenvalue, say the fundamental tone, i.e., the first eigenvalue, on the order of the operator. In fact,
this will now depend in a critical
fashion on whether Navier or Dirichlet boundary conditions are being considered -- see Theorem~\ref{thmA} and Remark~\ref{naviervsdiriichlet}. To the best of our knowledge,
so far the bounds addressing this issue did not distinguish between these two types of boundary conditions either -- see~\cite{erve,puccserr}, for instance.

One of the main purposes of this paper is thus to fill in this gap in the literature, by determining the precise asymptotic behaviour of
$\left(\diri{k}{m,t}(\Omega)\right)^{1/(2t)}$,
with respect to the parameters $m$ and $t$, as $m$ gets large. In fact, this is a consequence of stronger results concerning the order of growth of $\diri{k}{m,t}(\Omega)$ 
in $m$, when either $t$ or the difference $m-t$ is kept fixed (see Theorems~\ref{thmA1} and~\ref{ballasympt}, and Section~\ref{sec:ub}). To this end, let us recall the $\Theta$ notation, namely, $f(m)=\Theta(g(m))$ as $m$ goes to infinity
if there exist positive constants $c_1$ and $c_2$ such that $c_1g(m)\le f(m)\le c_2 g(m)$ for all sufficiently large $m$.
\begin{thmx}[Asymptotic behaviour for general domains]\label{maingeral}
Let $\Omega\subseteq\mathbb R^d$ be a bounded domain and $k\in\mathbb N$. Then, for $t$ fixed,
\[
\lambda_k^{(m,t)}(\Omega)=\Theta(m^{2t}), \mbox{ as }  m\to\infty,
\]
while for $h$ fixed,
\[
\Theta\left(2^{2m}e^{-2m}m^{2m-2h+\frac{d}{2}}\right) = \diri{1}{m-h}(\Omega)\leq \lambda_k^{(m,m-h)}(\Omega)=\bo(2^{2m}e^{-2m}m^{2m-h+d/2}),
\]
as  $m\to\infty$
\end{thmx}
\noindent As will be seen, our results point to this upper bound being sharp and thus the actual behaviour of $\diri{1}{m,m-h}(\Omega)$ being that indicated
by the term on the right-hand side above.
The proof of Theorem~\ref{maingeral} relies on a detailed study of the first eigenvalue of the unit ball which is, in turn, a consequence of 
an auxiliary problem associated with~\eqref{dirichletbuckling} -- see Section~\ref{sec:lowbounds}. We believe this auxiliary problem to be of interest in its own right, and will study it in more detail elsewhere.
 
 As a result of this approach, for the case of the polyharmonic operator ($m=t$)  we are able to derive sharper bounds for the first  eigenvalue, and thus obtain
 a more precise asymptotic behaviour (see Theorems~\ref{ballupperb} and~\ref{eineqaux}, and Lemma~\ref{lemeineqaux}).

\begin{thmx}[Bounds for the fundamental tone of the ball]\label{thmA} Let $\mathbb{B}$ be the unit ball in $\mathbb R^{d}$ and denote by $\lambda_{1}^{(m)}(\mathbb B)$ the first eigenvalue of
the problem
\begin{equation}
\label{polyh}
\left\{\begin{array}{ll}
(-\Delta)^m u=\lambda u, & {\rm \ in\ }\mathbb{B},\eqskip
u=\fr{\partial u}{\partial \nu}=\dots=\fr{\partial^{m-1} u}{\partial \nu^{m-1}}=0, & {\rm \ on\ }\partial \mathbb{B}.
\end{array}\right.
\end{equation}
Then
\begin{multline*}
\fr{2\pi}{\Gamma(d/2)} \left(\fr{2m}{e}\right)^{2m}  m^{\frac{d}{2}}
\approx 2^{2m}\fr{\Gamma(m+1)\Gamma(m+d/2)}{\Gamma(d/2)} 
\leq \diri{1}{m}(\mathbb B) \\
\le 2^{2m}\fr{\Gamma^2(m+1)\Gamma(2m+d/2+1)}{(m+d/2)\Gamma(d/2)\Gamma(2m+1)} 
\approx 2^{\frac d 2}\fr{2\pi}{\Gamma(d/2)} \left(\fr{2m}{e}\right)^{2m} m^{\frac{ d}{ 2}},
\end{multline*}
where $\Gamma$ is the Euler Gamma function and the indicated asymptotics refer to large $m$.
\end{thmx}
Since the asymptotic behaviour for large $m$ of the quotient between the upper and lower bounds equals $2^{d/2}$ and thus depends
only on the dimension $d$, the above expressions yield the right order of growth of the first eigenvalue with respect to $m$.

\begin{remark}\label{naviervsdiriichlet}
 An immediate consequence of these results is that the growth of the first eigenvalue of problem~\eqref{polyh} is much faster than that of the corresponding
 problem with the Navier boundary conditions
 \[
  u=\Delta u=\dots=\Delta^{m-1}u=0.
 \]
 In this case, the first eigenvalue of the unit ball equals $j_{d/2-1,1}^{2m}$, where $j_{d/2-1,1}$ denotes the first positive zero of the Bessel function of
 the first kind $J_{d/2-1}$. 
 A very different behaviour between general polyharmonic operators with Dirichlet and Navier boundary conditions has also been identified
 recently for the corresponding regularised spectral determinants~\cite{freilipo}.
\end{remark}
\begin{remark}\label{psrem}
To the best of our knowledge, the best lower bound in the literature for the first eigenvalue with Dirichlet boundary conditions was that given in~\cite{puccserr},
 which stated that $\diri{1}{m}(\mathbb B)\geq (\lambda_{1}\mu_{2})^{m/2}$ for even $m$, and $\diri{1}{m}(\mathbb B)\geq \lambda_{1} (\lambda_{1}\mu_{2})^{(m-1)/2}$
 for odd $m$, where $\lambda_{1}$ and $\mu_{2}$ are the first Dirichlet and first nontrivial Neumann eigenvalues of the Laplace operator in $\mathbb{B}$,
 respectively.
\end{remark}

Considering now the behaviour of the first eigenvalue normalised with respect to the order of the operator in the way described above, we see that Theorem~\ref{thmA}
yields the following precise two-term asymptotic behaviour.
\begin{corx}\label{normgrowth}
\[
\left[\diri{1}{m}(\mathbb B)\right]^{\fr{1}{2m}}=\fr{2m}{e}+\fr{d}{2e} \log m +\bo(1), \hspace*{10mm} \mbox{ as } m\to\infty. 
\]
\end{corx}

\begin{remark}
 \label{remerve} In one dimension, i.e., on the interval $(-1,1)$, the quantity $\gamma_{m} = \left[ \diri{1}{m}(-1,1) \right]^{\frac1{2m}}$ appears in~\cite{erve} in
 connection with an observability problem in control, where it is shown that $\gamma_{m}\leq \fr{\pi m}{2}$. Theorem~\ref{thmA} yields 
 \[
 \gamma_{m} \leq \fr{1}{2} \left[\fr{(m!)^2 (4 m+1)!}{(2 m)! (2 m+1)!} \right]^{1/(2m)},
 \]
 and it is possible to see that, except for the case where $m$ equals one, this value is smaller than $\pi m/2$. Furthermore, by Corollary~\ref{normgrowth} we
 have
 \[
  \gamma_{m} = \fr{2m}{e}+\frac{1}{2e} \log m +\bo(1), \mbox{ as } m\to\infty,
 \]
 answering in the negative the question raised in~\cite{erve} as to whether the quantity ${\ds \liminf_{m\to\infty}}\, \gamma_{m}$ might be finite or not.
\end{remark}

The above results do not take into consideration the normalisation with respect to volume, which is natural, and relevant in some contexts such as when dealing
with isoperimetric inequalities. In order to do this, we should consider the quantity $\left|\Omega\right|^{2m/d} \diri{1}{m}(\Omega)$ instead. This would give, for
instance, a generalised Rayleigh-Faber-Krahn conjecture of the form
\begin{conj*}[Generalised Rayleigh-Faber-Krahn inequality]
 For any domain $\Omega$ of finite measure the first eigenvalue of problem~\eqref{dirichletbuckling} satisfies
 \[
  \left|\Omega\right|^{2t/d} \diri{1}{m,t}(\Omega) \geq \omega_{d}^{2t/d} \diri{1}{m,t}(\mathbb B).
 \]
\end{conj*}
We note that the only cases where the above conjecture has been shown to hold are the classical Faber-Krahn
result for the Laplacian ($m=t=1$)~\cite{fab,kra1,kra2},
and in dimensions two and three when $m=t=2$~\cite{ashben,nad}, while in general only the criticality of the ball
has been proved~\cite{bula}. Existence of an optimal domain within the class of open sets in
dimensions up to $4m$ has recently been proved in~\cite{leyl}.
We note that, however, at this stage the dependence of the asymptotic behaviour of the normalized eigenvalue $\left|\Omega\right|^{2m/d} \diri{1}{m}(\Omega)$
on the volume of the domain $\Omega$ is not completely clear. More precisely, we show that there are domains with different volumes for which the first term in
the corresponding asymptotic expansion is the same. To see this, we consider another specific domain, now the $d-$hypercube, for which we obtain the same first term
in the asymptotics for the first eigenvalue (see Theorem~\ref{hrect}).
\begin{thmx}\label{thmD} Let $Q\subseteq\mathbb R^d$ be the hypercube $(-1,1)^d$. Then
$$
\left[\diri{1}{m}(Q)\right]^{\frac{1}{2m}}=\fr{2m}{e}+\bo(\log m), \hspace*{10mm} \mbox{ as } m\to\infty. 
$$
\end{thmx}
Note that, not only do the domains $\mathbb B$ and $Q$ have different volumes for $d$ larger than one, but while the volume of the ball of unit radius goes
to zero as the dimension becomes large, that of the hypercube equals $2^d$ and is thus unbounded. On the other hand, it is obvious that rescaling either of
these domains by a factor $\ab$ will change the first term in the above asymptotics to $2m /(e\ab)$, clearly showing that this term is not independent of the domain.

A closer analysis of the asymptotic behaviour of the eigenvalues shows that Theorem~\ref{thmD} is in fact valid for any hyperrectangle where the length of
the smallest side is two (see Theorem~\ref{hrect}). These domains all have the same inradius as $\mathbb B$ and $Q$, suggesting that this might
be the relevant geometric quantity appearing in the first term of the asymptotic expansion. Indeed, it is possible to argue that,  for a general domain and as the order
$m$ increases towards infinity, in principle the eigenfunctions become more and more regular at the boundary, since for any $k\in\mathbb N$ there will always
be an index $\tilde m$ for which the eigenfunctions will be differentiable at least $k$-times for any $m\ge \tilde m$. In addition, eigenfunctions are required to have more
and more derivatives vanishing at the boundary, since they are $H^m_0$-functions. All
these observations could be used to argue that the sequence of eigenvalues $\diri{1}{m}(\Omega)$ is asymptotically highlighting the ``smallest dimension'' of the
domain, i.e., its inradius, since there the eigenfunctions will be ``less free than along longer directions'', in some sense. Based on this we propose the following
\begin{conj*}
 For any domain $\Omega$ of finite measure and with inradius $\ab$ the following asymptotic expansion holds
$$
\left[\diri{1}{m}(\Omega)\right]^{\frac{1}{2m}}=\fr{2m}{e\ab}+\bo(\log(m)), \hspace*{10mm} \mbox{ as } m\to\infty. 
$$
\end{conj*}
An interesting family of domains where to test such a claim would be annuli and spherical shells, where the spectrum can be characterized somewhat explicitly by
means of Bessel functions. However, in that case the difficulty lies in the identification of the first eigenvalue, and its dependence on $m$ as it goes to infinity
(cf.\ \cite{buosoparini}).

Another set of natural questions related to the eigenvalue problem~\eqref{dirichletbuckling} is how the different eigenvalues are related between themselves
for different values of the order $k$ of the eigenvalue and the parameters $m$ and $t$. Inequalities for the cases where $m=1,2$ have been studied by several authors
and are by now classical (cf.\ \cite{blps,liuineq}). Here we consider the more general case where $m$ and $t$ are positive integers with $t\leq m$,
showing that there is a hierarchy between these eigenvalues and leading to a number of interesting relations depending on how these indices vary. In fact, this
turns out to be one of the ingredients we use to obtain Theorem~\ref{thmA} (see Section~\ref{sec.ineq}).
\begin{thmx}\label{thmE}
For any domain $\Omega$ in $\mathbb R^d$ of finite measure and $k,t,s,m\in\mathbb N$ such that $t\le s\le m$, the following inequalities hold:
 
 \begin{itemize}
 \item[i.] $\left[\diri{k}{m,t}(\Omega)\right]^{1/t}\leq\left[\diri{k}{m+1,t+1}(\Omega)\right]^{1/(t+1)}$\vspace*{2mm}
 \item[ii.] $\diri{k}{m,t}(\Omega)\leq\diri{k}{m+1,t}(\Omega)$\vspace*{2mm}
 \item[iii.] $\left[\diri{k}{m,s}(\Omega)\right]^{1/s}\le \left[\diri{k}{m,t}(\Omega)\right]^{1/t}$\vspace*{2mm}
 \end{itemize}
\end{thmx}

\begin{remark}
Note that, in view of Friedlander's inequality $\lambda_1\ge\mu_2$ \cite{fried}, inequality {\rm i.} in Theorem~\ref{thmE} is in fact a generalization (and an improvement) of the bound for $\diri{1}{m}(\Omega)$
provided in~\cite{puccserr} -- see Remark~\ref{psrem} 
and inequality~\eqref{secondinequality} with $t=m$. We highlight the fact that
the bound of~\cite{puccserr} is obtained by iterating Poincar\'e-type inequalities involving only Laplacian eigenvalues, while our technique
employs a direct manipulation of the Rayleigh quotients by means of the Cauchy--Schwarz inequality.
\end{remark}

In particular, one classical inequality we consider is the following, originally due to Payne~\cite{payne},
which in our notation reads as
\begin{equation}
\label{payneineq11}
\diri{2}{1,1}(\Omega)\le \diri{1}{2,1}(\Omega).
\end{equation}
Generalisations of this inequality have been considered in several directions, both with respect to higher order eigenvalues (cf.\ \cite{blps,fried04,friedp,liup}) and in
terms of related operators (cf.\ \cite{ilsho,kell}). Here we provide a version that considers polyharmonic eigenvalues of the type \eqref{dirichletbuckling}.
As for higher order eigenvalues, the analogue of inequality \eqref{payneineq11} is known not to hold for intervals in $\mathbb R$ (see \cite[Section 4]{blps}). We provide here an alternative proof that also encompasses the case $m\ge2$.

The analysis of problem~\eqref{dirichletbuckling} is complemented with the full characterization of its eigenvalues and eigenfunctions on the ball. For this
specific domain the eigenfunctions are known to be decomposable into a radial part, consisting of a combination of Bessel functions and (poly)harmonic polynomials,
and a spherical harmonic function. This allows us to write explicitly the defining equations for the eigenvalues
for any choice of $m$ and $t$, including the dependence of these equations
on the degree $\ell\in\mathbb N_0$ of the spherical (poly)harmonic functions. As an immediate consequence we observe that, when $t=1$, the
eigenvalues are (squares of) zeroes of Bessel functions, just like for the Dirichlet Laplacian (see Corollary~\ref{buckevs}). On the other hand, for $t>1$ the
equation becomes more involved, and it is not clear a priori if there are infinitely many solutions for any $\ell$, nor how such values interlace. 
In the particular case where $t=m$, it is also known that $\diri{1}{m}(\mathbb B)$ is simple and the associated eigenfunction may be chosen to be positive, while
for $t<m$ the simplicity of the first eigenvalue remains an open problem.

We should also remark that this behaviour is different from what happens in the case of the hypercube $Q$ mentioned above, for which no separation of variables is
possible (except for the Laplacian) and any eigenfunction will instead necessarily be sign-changing, making the corresponding
analysis more difficult to be carried out (cf.\ ~\cite{kkm}; see also \cite{buofre} for considerations on isoperimetric inequalities on rectangles).

The structure of the paper is as follows. In Section~\ref{sec.ineq} we state and prove several general inequalities for different combinations of the parameters
$m$ and $t$ and the order of the eigenvalues. In Section~\ref{genasympt} we provide upper and lower bounds for $\diri{1}{m,t}(\mathbb{B})$ which allow us to study its asymptotic behaviour with respect to the differential orders involved and, in particular, prove some of the main results above. The key ingredients in this process are the choice of an appropriate family of test functions and the consideration of an auxiliary eigenvalue problem. Section~\ref{payneineq} is dedicated to Payne's inequality \eqref{payneineq11} and its generalization for problem \eqref{dirichletbuckling}. Finally, Appendix~\ref{ballsection} contains a systematic study of the eigenfunctions and eigenvalues of the ball, while Appendix~\ref{techcompt} contains various technical lemmas.

\section{General inequalities}
\label{sec.ineq}

In this section we provide a number of inequalities relating the eigenvalues $\diri{k}{m,t}$ for different values of $m$ and $t$. In particular,
Theorems~\ref{increasing},~\ref{sameorder} and~\ref{upperthm} contain the proofs of the three inequalities in Theorem~\ref{thmE}.

To this end, we recall that the weak formulation of the problem defined by~\eqref{dirichletbuckling} is given by
\begin{equation}
\label{first}
\int_{\Omega}D^m u: D^m\phi=\diri{ }{m,t}\int_\Omega D^{m-t} u: D^{m-t}\phi,\ \forall \phi\in H^m_0(\Omega),
\end{equation}
where $D^m u: D^m\phi$ denotes the Frobenius product (we recall that $\alpha !=\alpha_1 !\cdot\alpha_2 !\cdots\alpha_d!$)
$$
D^m u: D^m\phi=\sum_{|\alpha|=m}\frac{m!}{\alpha !}\fr{\partial^\alpha u}{\partial x^\alpha}\fr{\partial^\alpha \phi}{\partial x^\alpha}.
$$
Note that equation \eqref{first} is associated with a self-adjoint operator with compact resolvent for any domain $\Omega$ of finite measure in $\mathbb R^d$,
so that its spectrum is discrete and consists of a diverging sequence of positive eigenvalues of finite multiplicities. In particular, the eigenfunctions form an orthogonal basis for both $H^m_0(\Omega)$ and $H^{m-t}_0(\Omega)$. Moreover, if we set
\begin{equation*}
\nabla^m \phi=
\left\{\begin{array}{ll}
\Delta^k\phi,&\text{if\ }m=2k,\eqskip
\nabla\Delta^k\phi,&\text{if\ }m=2k+1,
\end{array}\right.
\end{equation*}
then, using the fact that the ambient space is $H^m_0(\Omega)$, we have that equation \eqref{first} is equivalent to
\begin{equation*}
\int_{\Omega}\nabla^m u\cdot \nabla^m\phi=\diri{}{m,t}\int_\Omega \nabla^{m-t}u\cdot\nabla^{m-t}\phi,\ \forall \phi\in H^m_0(\Omega).
\end{equation*}
In this notation, the index $m$ represents the highest differential order appearing in the operator, while the index $t$ is the difference between the differential
orders of the two operators involved. Because of this, $t$ is the relevant index to be considered when studying the asymptotic properties of eigenvalues,
as may be seen in the asymptotic expression~\eqref{weylmt}. For this reason, we will refer to the index $t$ as half the Weyl order of the eigenvalues,
while the index $m$ will be referred to as half the differential order.

We also recall that, for any positive integers $m$ and $t$ with $1\le t\le m$, the sequence of eigenvalues of problem~\eqref{dirichletbuckling}
may be characterized variationally as follows
\begin{equation}\label{rayleighbuckling}
\diri{k}{m,t}(\Omega)=\min_{\substack{V\subset H^m_0(\Omega)\\ \dim V=k}}\ \max_{0\neq u\in V}\fr{\int_\Omega |\nabla^m u|^2}{\int_\Omega |\nabla^{m-t}	u|^2}.
\end{equation}

In addition, we note that the characterization \eqref{rayleighbuckling} yields the following rescaling property
\begin{equation}
\label{scaling}
\lambda_k^{(m,t)}(\varrho\Omega)=\varrho^{-2t}\lambda_k^{(m,t)}(\Omega),
\end{equation}
for any $\varrho>0$.

Throughout, we will make extensive use of the following inequality 
\begin{equation}
\label{basic}
\int_\Omega |\nabla^m u|^2=\left|\int_\Omega\nabla^{m-1}u\cdot\nabla^{m+1}u\right|\leq\left(\int_\Omega|\nabla^{m-1}u|^2\right)^{\frac{1}{2}}
\left(\int_\Omega|\nabla^{m+1}u|^2\right)^{\frac{1}{2}},
\end{equation}
which is a direct consequence from the Cauchy--Schwarz inequality and is valid for any function $u\in H^{m+1}(\Omega)\cap H^m_0(\Omega)$.


We can now show that the eigenvalues are monotone with respect to the order $t$.
\begin{thm}
\label{increasing}
For any $k,t,m \in\mathbb N$, $t\le m\in\mathbb N$, and for any domain $\Omega\in\mathbb R^d$ of finite measure,
\begin{equation}
\label{secondinequality}
\left[\diri{k}{m,t}(\Omega)\right]^{\fr{1}{t}}\leq\left[\diri{k}{m+1,t+1}(\Omega)\right]^{\fr{1}{t+1}}.
\end{equation}
\end{thm}
%

\begin{proof}
We will first show that, for a fixed $h\le m$, a function $u\in H^{m+1}(\Omega)\cap H^m_0(\Omega)$ satisfies
\begin{equation}
\label{firststep}
\left(\fr{\int_\Omega |\nabla^m u|^2}{\int_\Omega |\nabla^{h} u|^2}\right)^{\fr 1 {m-h}}\leq\left(\fr{\int_\Omega|\nabla^{m+1}u|^2}{\int_\Omega |\nabla^{h} u|^2}\right)^{\fr 1 {m-h+1}}.
\end{equation}
From inequality~\eqref{basic} we have
\begin{equation}
\label{firststep1}
\fr{\int_\Omega |\nabla^m u|^2}{\int_\Omega |\nabla^{h} u|^2}\leq\left(\fr{\int_\Omega|\nabla^{m-1}u|^2}{\int_\Omega |\nabla^{h} u|^2}\right)^{\fr 1 2}
\left(\fr{\int_\Omega|\nabla^{m+1}u|^2}{\int_\Omega |\nabla^{h} u|^2}\right)^{\fr 1 2},
\end{equation}
so that when $m-1=h$ claim~\eqref{firststep} follows. Otherwise, arguing by induction on $m$, inequality~\eqref{firststep1} yields
\begin{equation*}
\fr{\int_\Omega |\nabla^m u|^2}{\int_\Omega |\nabla^{h} u|^2}\leq\left(\fr{\int_\Omega|\nabla^{m}u|^2}{\int_\Omega |\nabla^{h} u|^2}\right)^{\fr{m-h-1}{2m-2h}}
\left(\fr{\int_\Omega|\nabla^{m+1}u|^2}{\int_\Omega |\nabla^{h} u|^2}\right)^{\frac 1 2},
\end{equation*}
which again implies \eqref{firststep}.

It is clear that inequality \eqref{firststep} yields
\begin{equation}
\label{onlymax}
\max_V\left(\fr{\int_\Omega |\nabla^m u|^2}{\int_\Omega |\nabla^hu|^2}\right)^{\fr 1 {m-h}}\leq\max_V\left(\fr{\int_\Omega|\nabla^{m+1}u|^2}
{\int_\Omega |\nabla^hu|^2}\right)^{\fr 1 {m-h+1}},
\end{equation}
for any subspace $V$ of $H^{m+1}_0(\Omega)$ (which is a subspace of $H^{m+1}(\Omega)\cap H^m_0(\Omega)$) of given finite dimension. Taking now the infimum of~\eqref{onlymax}
over subspaces of dimension $k$ and substituting $h=m-t$ gives \eqref{secondinequality}, since the exponents can be taken out of the minimax.
\end{proof}

As already stated, Theorem~\ref{increasing} deals with chains of eigenvalues where the Weyl order increases. If instead we consider eigenvalues with the
same Weyl order but where the differential order increases, we obtain the following

\begin{thm}
\label{sameorder}
For any $k,t,m\in\mathbb N$, $t\le m\in\mathbb N$, and for any domain $\Omega\in\mathbb R^d$ of finite measure,
\begin{equation*}
\diri{k}{m,t}(\Omega)\leq\diri{k}{m+1,t}(\Omega).
\end{equation*}
\end{thm}


\begin{proof}
From \eqref{basic} we get that a function $u\in H^{m+1}(\Omega)\cap H^m_0(\Omega)$ satisfies
\begin{equation}
\label{firststepsame1}
\fr{\int_\Omega |\nabla^h u|^2}{\int_\Omega |\nabla^{h-1} u|^2}\leq\fr{\int_\Omega|\nabla^{h+1}u|^2}{\int_\Omega |\nabla^h u|^2},
\end{equation}
for any $h\le m$. By iterating inequality \eqref{firststepsame1} for $m-t+1\le h\le m$ we obtain
\begin{equation*}
\fr{\int_\Omega |\nabla^m u|^2}{\int_\Omega |\nabla^{m-t} u|^2}\leq\fr{\int_\Omega|\nabla^{m+1}u|^2}{\int_\Omega |\nabla^{m-t+1} u|^2}.
\end{equation*}
The proof of the theorem can now be done following that of Theorem \ref{increasing}.
\end{proof}

We have now a generalization of a claim made by Payne~\cite[formula (63)]{payne}.

\begin{thm}
\label{lower}
For any $k,t,m\in\mathbb N$, $t\le m\in\mathbb N$, and for any domain $\Omega\in\mathbb R^d$ of finite measure,
\begin{equation}
\label{generalizedpayne}
\diri{k}{m,t}(\Omega)\ge\diri{k}{m,1}(\Omega)\dprod_{h=m-t+1}^{m-1}\diri{1}{h,1}(\Omega). 
\end{equation}
\end{thm}


\begin{proof}
It suffices to observe that the Rayleigh quotient of $\diri{k}{m,t}(\Omega)$ can be split as follows 
$$
\fr{\int_\Omega|\nabla^m u|^2}{\int_\Omega |\nabla^{m-t} u|^2}=\dprod_{h=m-t+1}^{m}\fr{\int_\Omega|\nabla^h u|^2}{\int_\Omega |\nabla^{h-1}u|^2}
\ge  \fr{\int_\Omega|\nabla^m u|^2}{\int_\Omega |\nabla^{m-1}u|^2}\dprod_{h=m-t+1}^{m-1}\diri{1}{h,1}(\Omega),
$$
where in the last inequality we used the fact that
$$
\diri{1}{h,1}(\Omega)=\inf_{0\neq u\in H^m_0(\Omega)} \fr{\int_\Omega|\nabla^h u|^2}{\int_\Omega |\nabla^{h-1}u|^2}
=\min_{0\neq u\in H^h_0(\Omega)} \fr{\int_\Omega|\nabla^h u|^2}{\int_\Omega |\nabla^{h-1}u|^2}.
$$
The result now follows applying the minimax.
\end{proof}

While in~\cite{payne} there are some considerations on how sharp the inequality is for the case $m=2,k=1$, it is not clear how good the bound~\eqref{generalizedpayne}
is in the general case.

As a counterpart to Theorem \ref{lower} we have the following generalization of an inequality between the Dirichlet bilaplacian and buckling eigenvalues
(see e.g., \cite{liuineq}; see also \cite[Remark 2.12]{blps}).
\begin{thm}
\label{upperthm}
For any domain $\Omega$ in $\mathbb R^d$ of finite measure, for any $k,m\in\mathbb N$, and for any $1\le t\le s\le m$
\begin{equation}
\label{bucklingboundgen}
\left[\diri{k}{m,s}(\Omega)\right]^{\frac{1}{s} }\le \left[\diri{k}{m,t}(\Omega)\right]^{\frac{1}{t} }. 
\end{equation}
In particular,
\begin{equation}
\label{bucklingbound}
\diri{k}{m}(\Omega)\le \left[\diri{k}{m,1}(\Omega)\right]^m. 
\end{equation}
\end{thm}
\begin{proof}
The proof is based on the Gagliardo-Nirenberg-type inequality of Theorem \ref{gntineqthm} below. From \eqref{gntineq} we get
$$
\left(\fr{\int_\Omega |\nabla^m u|^2}{\int_\Omega |\nabla^{m-s} u|^2}\right)^{\fr 1 s}\leq\left(\fr{\int_\Omega|\nabla^m u|^2}{\int_\Omega |\nabla^{m-t} u|^2}\right)^{\fr 1 t}.
$$
Taking the standard minimax in $H^m_0(\Omega)$ and observing that the exponents can be taken out of the minimax, we obtain \eqref{bucklingboundgen}. Finally, equation \eqref{bucklingbound} follows from \eqref{bucklingboundgen} with $t=1$ and $s=m$.
\end{proof}

\begin{remark}
We observe that combining Theorems~\ref{increasing} and~\ref{upperthm} we get and alternative proof for Theorem \ref{sameorder}: in fact, considering $s=t+1$,
$$
\left[\diri{k}{m,t}(\Omega)\right]^{\frac{1}{t} }\le \left[\diri{k}{m+1,t+1}(\Omega)\right]^{\frac{1}{t+1} }\le \left[\diri{k}{m+1,t}(\Omega)\right]^{\frac{1}{t} }. 
$$
\end{remark}

We now provide a very general Gagliardo-Nirenberg-type inequality for polyharmonic operators on $H^m_0(\Omega)$.

\begin{thm}
\label{gntineqthm}
Let $m,t,s\in\mathbb N$ be such that $0\le t\le s\le m$, and let $u\in H^m_0(\Omega)$. Then
\begin{equation}
\label{gntineq}
\left(\int_\Omega |\nabla^{m-t} u|^2\right)^{s}\le \left(\int_\Omega |\nabla^{m-s} u|^2\right)^{t}\left(\int_\Omega |\nabla^m u|^2\right)^{s-t}.
\end{equation}
Equivalently, for any $0\le q\le p\le m$,
\begin{equation*}
\left(\int_\Omega |\nabla^{p} u|^2\right)^{m-q}\le \left(\int_\Omega |\nabla^{q} u|^2\right)^{m-p}\left(\int_\Omega |\nabla^m u|^2\right)^{p-q}.
\end{equation*}
\end{thm}
\begin{remark}
We observe two important points that make Theorem \ref{gntineqthm} peculiar: on the one hand, it does not involve generic derivatives but specific combinations involving
polyharmonic operators, and it specifically applies to $H^m_0$-functions only; on the other hand, there are no unknown constants depending on the indices or on the
geometry of $\Omega$.
\end{remark}

\begin{proof}
The proof is quite involved since it is done by induction on the three indices $t,s,m$. We divide it into 5 steps.

{\bf Step 1: $t=1$, $s=m$.} We want to prove here that
\begin{equation}
\label{caso1-1}
\left(\int_\Omega |\nabla^{m-1} u|^2\right)^{m}\le \left(\int_\Omega |u|^2\right)\left(\int_\Omega |\nabla^m u|^2\right)^{m-1}.
\end{equation}
We proceed by induction on $m$. Clearly for $m=1$ the inequality is trivial, while for $m=2$ it immediately follows from \eqref{basic}. Again by induction on~\eqref{basic} we have
\begin{equation}
\label{caso1-2}
\begin{array}{lll}
\left(\dint_\Omega |\nabla^{m} u|^2\right)^{2} & \le & \left(\dint_\Omega |\nabla^{m-1}u|^2\right)\left(\dint_\Omega |\nabla^{m+1} u|^2\right)\eqskip
& \le & \left(\dint_\Omega |u|^2\right)^{\fr 1 m}\left(\dint_\Omega |\nabla^m u|^2\right)^{\fr{m-1}{m}}\left(\dint_\Omega |\nabla^{m+1} u|^2\right).
\end{array}
\end{equation}
Rearranging the terms in \eqref{caso1-2} we obtain \eqref{caso1-1} where $m$ is replaced by $m+1$, concluding the induction.

{\bf Step 2: $t=2$, $s=m$.} This is similar to step 1: we want to prove that
\begin{equation}
\label{caso2-1}
\left(\dint_\Omega |\nabla^{m-2} u|^2\right)^{m}\le \left(\dint_\Omega |u|^2\right)^2\left(\dint_\Omega |\nabla^m u|^2\right)^{m-2}.
\end{equation}
We proceed by induction on $m$. Clearly for $m=2$ the inequality is trivial, while for $m=3$ it immediately follows from \eqref{basic}. Once more by induction on~\eqref{basic} we have
\begin{equation}
\label{caso2-2}
\begin{array}{lll}
\left(\dint_\Omega |\nabla^{m-1} u|^2\right)^{2} & \le & \left(\dint_\Omega |\nabla^{m-2}u|^2\right)\left(\dint_\Omega |\nabla^{m} u|^2\right)\eqskip
& \le & \left(\dint_\Omega |u|^2\right)^{\fr 2 m}\left(\dint_\Omega |\nabla^m u|^2\right)^{2\fr{m-1}{m}}\eqskip
& \le & \left(\dint_\Omega |u|^2\right)^{\fr 2 m}\left(\dint_\Omega |\nabla^{m-1} u|^2\right)^{\fr{m-1}{m}}\left(\dint_\Omega |\nabla^{m+1} u|^2\right)^{\fr{m-1}{m}}.
\end{array}
\end{equation}
Rearranging the terms in \eqref{caso2-2} we obtain \eqref{caso2-1} where $m$ is replaced by $m+1$, concluding the induction.

{\bf Step 3: $s=m$, general $t$.} Here we want to prove that
\begin{equation}
\label{caso3-1}
\left(\int_\Omega |\nabla^{m-t} u|^2\right)^{m}\le \left(\int_\Omega |u|^2\right)^{t}\left(\int_\Omega |\nabla^m u|^2\right)^{m-t}.
\end{equation}
We proceed by a double induction on $t$ and $m$ as follows. We take steps 1 and 2 as bases together with the case $t=m$ (which is trivial). We
therefore assume that formula \eqref{caso3-1} is valid for any $m$ for $\bar t< t$ and for $\bar m\le m$ for $\bar t=t$. We have
\begin{equation}
\label{caso3-2}
\left(\int_\Omega |\nabla^{m+1-t} u|^2\right)^{2}\le \left(\int_\Omega |\nabla^{m-t}u|^2\right)\left(\int_\Omega |\nabla^{m+2-t} u|^2\right).
\end{equation}
Considering $m+2-t=m-(t-2)$, by induction on $t$ inequality~\eqref{caso3-2} becomes
\begin{multline*}
\begin{array}{lll}
\left(\int_\Omega |\nabla^{m+1-t} u|^2\right)^{2} &
\le & \left(\int_\Omega |u|^2\right)^{\fr{t}{m}}\left(\int_\Omega |\nabla^{m} u|^2\right)^{\fr{m-t}{m}}
\left(\int_\Omega |u|^2\right)^{\fr{t-2}{m}}\left(\int_\Omega |\nabla^{m} u|^2\right)^{\fr{m-t+2}{m}}\eqskip
& = &
\left(\int_\Omega |u|^2\right)^{\fr{2t-2}{m}}\left(\int_\Omega |\nabla^{m} u|^2\right)^{2\fr{m-t+1}{m}}\eqskip
& \le & \left(\int_\Omega |u|^2\right)^{\fr{2t-2}{m}}\left(\int_\Omega |\nabla^{m-1} u|^2\right)^{\fr{m-t+1}{m}} \left(\int_\Omega |\nabla^{m+1} u|^2\right)^{\fr{m-t+1}{m}}.
\end{array}
\end{multline*}
Now, by considering $m-1=(m+1)-(2)$ and applying inequality \eqref{caso2-1} we see that
$$
\left(\int_\Omega |\nabla^{m-1} u|^2\right)^{m+1}\le \left(\int_\Omega |u|^2\right)^2 \left(\int_\Omega |\nabla^{m+1} u|^2\right)^{m-1},
$$
yielding
\begin{equation}
\label{caso3-4}
\left(\int_\Omega |\nabla^{m+1-t} u|^2\right)^{2}\le \left(\int_\Omega |u|^2\right)^{\fr{2t}{m+1}}\left(\int_\Omega |\nabla^{m+1} u|^2\right)^{2\fr{m-t+1}{m+1}}.
\end{equation}
Rearranging the terms in \eqref{caso3-4} we obtain \eqref{caso3-1} where $m$ is replaced by $m+1$, concluding the induction.

{\bf Step 4: $t=1$, general $s$ and $m$.} We want to prove that
\begin{equation}
\label{caso4-1}
\left(\int_\Omega |\nabla^{m-1} u|^2\right)^{s}\le \left(\int_\Omega |\nabla^{m-s}u|^2\right)\left(\int_\Omega |\nabla^m u|^2\right)^{s-1}.
\end{equation}
Inequality \eqref{caso4-1} is trivial for $s=1,2$ for any $m$. In order to proceed by induction we rewrite equation \eqref{caso4-1} as follows
\begin{equation}
\label{caso4-2}
\left(\int_\Omega |\nabla^{m-1} u|^2\right)^{m-p}\le \left(\int_\Omega |\nabla^{p}u|^2\right)\left(\int_\Omega |\nabla^m u|^2\right)^{m-p-1},
\end{equation}
which we want to prove for any $0\le p< m$. We already know that \eqref{caso4-2} is valid for any $p\in\mathbb N$ and for $m=p+1,p+2$. By induction on $m$
\begin{equation}
\label{caso4-3}
\begin{array}{lll}
\left(\int_\Omega |\nabla^{m} u|^2\right)^{2} & \le & \left(\int_\Omega |\nabla^{m-1}u|^2\right)\left(\int_\Omega |\nabla^{m+1} u|^2\right)\eqskip
& \le & 
\left(\int_\Omega |\nabla^{p} u|^2\right)^{\fr 1 {m-p}}\left(\int_\Omega |\nabla^m u|^2\right)^{\fr{m-p-1}{m-p}}\left(\int_\Omega |\nabla^{m+1} u|^2\right).
\end{array}
\end{equation}
Rearranging the terms in \eqref{caso4-3} we obtain \eqref{caso4-2} where $m$ is replaced by $m+1$, concluding the induction.

{\bf Step 5: conclusion.} We now proceed to the general case \eqref{gntineq}. It is clear from step 1 that the inequality is true for any $t$ and for any $s$, and for
$m=s$. Also, from step 4 it is clear that it is valid for for $t=1$ and for any $s$ and $m$. Supposing now that the inequality is valid for any $s,m$ when $\bar t<t$ and
for any $s\le\bar m\le m$ for $\bar t=t$, observing that $(m+1)-t=m-(t-1)$, we have
\begin{equation}
\label{caso5-1}
\left(\int_\Omega |\nabla^{m+1-t} u|^2\right)^{s}\le \left(\int_\Omega |\nabla^{m+1-s}u|^2\right)^{s\fr{t-1}{s-1}}\left(\int_\Omega |\nabla^m u|^2\right)^{s\fr{s-t}{s-1}}.
\end{equation}
Similarly, since $m=(m+1)-1$, we have
\begin{equation}
\label{caso5-2}
\left(\int_\Omega |\nabla^{m} u|^2\right)^{s}\le \left(\int_\Omega |\nabla^{m+1-s}u|^2\right)\left(\int_\Omega |\nabla^{m+1} u|^2\right)^{s-1}.
\end{equation}
Combining \eqref{caso5-1} with \eqref{caso5-2} we obtain \eqref{gntineq} where $m$ is replaced by $m+1$, concluding the induction. This concludes the proof.
\end{proof}

\section{Asymptotics with respect to the order of the operator}
\label{genasympt}

In this section we provide asymptotic estimates for the eigenvalues of problem \eqref{dirichletbuckling} on balls  and on other domains by combining the inequalities of
Section~\ref{sec.ineq} with the description of the eigenvalues on balls (see Appendix~\ref{ballsection}).

In the unit ball $\mathbb B$, while the eigenvalue $\lambda_1^{(m,t)}(\mathbb B)$ is very difficult to compute in general (cf.\ Corollary \ref{buckevs}), the eigenvalue
$\lambda_1^{(m,1)}(\mathbb B)$ is a Bessel zero that can be easily identified. Indeed, using Corollary \ref{buckevs} and the properties of zeroes of Bessel functions,
it is immediate to see that
$$
\lambda_1^{(h,1)}(\mathbb B)=j_{\kappa(h),1}^2\ \ {\rm for}\ \kappa(h)=h+\fr d 2 -2,
$$
where $j_{\kappa,1}$ is the smallest zero of $J_\kappa$. Combining this with inequalities \eqref{generalizedpayne} and \eqref{bucklingboundgen}, and since
(see \cite[formula 10.21.40]{nist})
$
j_{\kappa,1} = \kappa+\bo(\kappa^{\frac 1 3}),      
$
we have proved that

\begin{thm}\label{thmA1}
For $d\ge 1$,
\begin{equation}
\label{gen.mt}
\dprod_{h=m-t}^{m-1}j_{h+\fr d 2 -2,1}^2 \le \lambda_1^{(m,t)}(\mathbb B)\le j_{m+\fr d 2-2,1}^{2t}.
\end{equation}
In particular, for $t$ fixed as $m\to\infty$ 
$$
\lambda_1^{(m,t)}(\mathbb B)=m^{2t}+\so(m^{2t}).
$$
\end{thm}

Notice that, in the case where $t=m-h$ with $h$ fixed, inequalities~\eqref{gen.mt} produce the following bounds
\begin{equation}
\label{imprecise}
\fr m e \lessapprox \left[\lambda_1^{(m,m-h)}(\mathbb B)\right]^{\fr 1 {2(m-h)}} \lessapprox  m,
\end{equation}
where the left-hand side comes from the combination of $
j_{\kappa,1} = \kappa+\bo(\kappa^{\frac 1 3}),      
$
with Stirling formula (see also formula~\eqref{stirlingnist}), while the right-hand side is immediate. Note that the bounds~\eqref{imprecise} provide the correct asymptotic order for  $\left[\lambda_1^{(m,m-h)}(\mathbb B)\right]^{\fr 1 {2(m-h)}}$, but they do not yield a precise constant. For this
reason we proceed with a refined analysis of this case.


\subsection{Improved eigenvalue upper bounds}\label{sec:ub}


We provide here an improved upper bound for $\diri{1}{m,m-h}(\mathbb B)$ by means of specific test functions.

\begin{thm}\label{ballupperb}
For any $m,h\in\mathbb N$ with $h<m$ we have
$$
\begin{array}{lll}
\diri{1}{m,m-h}(\mathbb B)& \le & \fr{2^{2m-2h}\Gamma^2(m-h+1)\Gamma(2m-h+1+d/2)}{(m+d/2)\Gamma(h+d/2)\Gamma(2m-2h+1)}\eqskip
& = & 2^{\frac d 2-h}\fr{2\pi}{\Gamma(h+d/2)} 2^{2m}  e^{-2m} m^{2m+\frac{ d}{ 2}-h}\left(1+\bo(m^{-1})\right),
\end{array}
$$
as $m\to\infty$.
\end{thm}
\begin{proof}
 The first inequality is obtained using $f(x)=(1-|x|^2)^m$ as a test function in the Rayleigh quotient,
and then using Lemma \ref{hyperlemma}.

In order to prove the asymptotic expression, we recall the Stirling approximation (see also \cite[Formula 5.11.8]{nist})
\begin{equation}
\label{stirlingnist}
\log \Gamma (z+h)=\left(z+h-\frac 1 2\right)\log z-z+\frac1 2\log(2\pi)+\bo(z^{-1}),
\end{equation}
or alternatively
\begin{equation}
\label{stirlingnist2}
\Gamma(z+h)=\sqrt{ 2 \pi}z^{z+h-\frac 1 2}e^{-z}\left(1+\bo(z^{-1})\right),
\end{equation}
where $h$ is a fixed value.

From the Stirling approximation~\eqref{stirlingnist2} we obtain
\begin{gather*}
 \begin{array}{lll}
  \fr{2^{2m-2h}\Gamma^2(m-h+1)\Gamma(2m-h+1+d/2)}{(m+d/2)\Gamma(h+d/2)\Gamma(2m-2h+1)}\eqskip
\hspace*{10mm}=\fr{2^{2m-2h}\left(\sqrt{2\pi}m^{m-h+1/2}e^{-m}\right)^2\sqrt{2\pi}(2m)^{2m-h+(d+1)/2}e^{-2m}}
{m\Gamma(h+d/2)\sqrt{2\pi}(2m)^{2m-2h+1/2}e^{-2m}} \left(1+\bo(m^{-1})\right)\eqskip
\hspace*{10mm}= 2^{\frac d 2 -h}\fr{2\pi}{\Gamma(d/2)} 2^{2m}  e^{-2m} m^{2m-h+\frac{ d}{ 2}}\left(1+\bo(m^{-1})\right).
 \end{array}
\end{gather*}
yielding the asymptotic behaviour indicated.
\end{proof}

\subsection{Eigenvalue lower bounds\label{sec:lowbounds}}

Let us consider the following problem

\begin{equation}
\label{auxbuck}
\begin{cases}
(-\Delta)^mu=\navi{}{m,t} (-1)^{m-t}\nabla^{m-t}\fr{1}{(1-|x|^2)^t}\nabla^{m-t}u, & {\rm \ in\ }\mathbb B,\eqskip
u=\fr{\partial u}{\partial \nu}=\dots=\fr{\partial^{m-1} u}{\partial \nu^{m-1}}=0, & {\rm \ on\ }\partial \mathbb B.
\end{cases}
\end{equation}
We observe that problem \eqref{auxbuck} admits the following weak formulation
\begin{equation*}
\int_{\mathbb B} \nabla^m u\nabla^m v=\navi{}{m,t} \int_{\mathbb B} \fr{1}{(1-|x|^2)^t}\nabla^{m-t}u\nabla^{m-t}v, \forall v\in \tilde H^{m,t}(\mathbb B),
\end{equation*}
where $\tilde H^{m,t}(\mathbb B)=H^m_0(\mathbb B)\cap H^{m-t}(\mathbb B,(1-|x|^2)^{-t})$. Notice that, since the weight $(1-|x|^2)^{-t}$ is not in $L^\infty(\mathbb B)$
we cannot automatically conclude that $\tilde H^{m,t}(\mathbb B)=H^m_0(\mathbb B)$; in general, it could be a strict subspace of $H^m_0(\mathbb B)$. Nevertheless, $\tilde H^{m,t}(\mathbb B)$
is a Hilbert space  containing $C^\infty_c(\mathbb B)$ as a dense set and,
in particular problem~\eqref{auxbuck} is associated with a self-adjoint operator with compact resolvent. Therefore, it admits a diverging sequence of positive eigenvalues of
finite multiplicity, whose associated eigenfunctions form an orthonormal basis for $\tilde H^{m,t}(\mathbb B)$. In addition, they can be defined variationally by
$$
\navi{k}{m,t}=\min_{\substack{V\subset \tilde H^{m,t}(\mathbb B)\\ \dim V=k}}\max_{0\neq u\in V}\fr{\int_{\mathbb B} |\nabla^m u|^2}{\int_{\mathbb B} |\nabla^{m-t} u|^2 (1-|x|^2)^{-t}}.
$$




\begin{thm}
For $m=t$, the fundamental tone of problem \eqref{auxbuck} is simple and is that identified in Lemma \ref{fund}, that is
$$
\navi{1}{m,m}=2^{2m}\fr{\Gamma(m+1)\Gamma(m+d/2)}{\Gamma(d/2)}.
$$
\end{thm}
\begin{proof}
First of all, we observe that when $m=t$ problem~\eqref{auxbuck} enjoys the so-called positivity preserving property, namely if $f$ is a positive
function then the solution $u$ of the problem
\begin{equation*}
\begin{cases}
(-\Delta)^mu=\fr{f}{(1-|x|^2)^{m}}, & {\rm \ in\ }\mathbb B,\eqskip
u=\fr{\partial u}{\partial \nu}=\dots=\fr{\partial^{m-1} u}{\partial \nu^{m-1}}=0, & {\rm \ on\ }\partial \mathbb B,
\end{cases}
\end{equation*}
is positive, i.e., $u>0$. This can be proved following verbatim the arguments of \cite[Section 5.1]{ggs}. This, in particular, allows us to mimic the arguments
in~\cite[Theorem 3.7]{ggs}, deducing that the first eigenvalue of problem \eqref{auxbuck} is simple and its associated eigenfunction is positive. Lemma~\ref{fund}
then leads to the conclusion.
\end{proof}

\begin{remark}\label{krarg}
We observe that, since for the case $m=t$ the use of a Krein-Rutman-type argument allows us to show that the first eigenfunction of the ball is positive,
it is expected that the case $m>t$ shows a similar behaviour, although it is not clear whether this argument could be adapted.
\end{remark}

We may thus use $\navi{1}{m,m}$ to obtain a lower bound for $\diri{1}{m}(\mathbb B)$.
\begin{thm}\label{eineqaux}
For any $m\in\mathbb N$, we have
\begin{equation}
\label{lb-est}
\diri{1}{m}(\mathbb B)\ge\navi{1}{m,m}=2^{2m}\fr{\Gamma(m+1)\Gamma(m+d/2)}{\Gamma(d/2)}.
\end{equation}
\end{thm}

\begin{proof}
We observe that, for any $u\in C^\infty_c(\mathbb B)$,
$
\int_{\mathbb B} u^2\le\int_{\mathbb B}\fr{u^2}{(1-|x|^2)^m},
$
from which we obtain
$$
\diri{1}{m}(\mathbb B)=\inf_{u\in C^\infty_c(\mathbb B)}\fr{\int_{\mathbb B} |\nabla^m u|^2}{\int_{\mathbb B} u^2}\ge\inf_{u\in C^\infty_c(\mathbb B)}\fr{\int_{\mathbb B}
|\nabla^m u|^2}{\int_{\mathbb B}\fr{u^2}{(1-|x|^2)^m}}=\navi{1}{m,m}.
$$
\end{proof}

\begin{remark}
Exploiting the argument used in the proof of Theorem \ref{eineqaux}, it is in fact possible to obtain the more general inequality
$$
\navi{k}{m,t}\le \diri{k}{m,t}(\mathbb B),\quad\forall m,t,k.
$$
\end{remark}

As in the case of the upper bound, we shall now determine the asymptotic behaviour of the expression obtained above.

\begin{lem}\label{lemeineqaux}
As $m\to\infty$ we have
\begin{equation}
\label{prelasympt2}
2^{2m}\fr{\Gamma(m+1)\Gamma(m+d/2)}{\Gamma(d/2)}\approx \frac{2\pi}{\Gamma(d/2)} 2^{2m}  e^{-2m} m^{2m+\frac{d}{2}}.
\end{equation}
\end{lem}

\begin{proof}
Using \eqref{stirlingnist2} we get
\[
\begin{array}{lll}
2^{2m}\fr{\Gamma(m+1)\Gamma(m+d/2)}{\Gamma(d/2)} & = & \frac{2^{2m}}{\Gamma(d/2)}\sqrt{2\pi}m^{m+1/2}e^{-m}\sqrt{2\pi}m^{m+(d-1)/2}e^{-m}\left(1+\bo(m^{-1})\right)\eqskip
& = & \fr{2\pi}{\Gamma(d/2)} 2^{2m}  e^{-2m} m^{2m+\frac{d}{2}}\left(1+\bo(m^{-1})\right),
\end{array}
\]
from which \eqref{prelasympt2} holds.
\end{proof}

\begin{cor}
We have
$$
\diri{1}{m}(\mathbb B) = \Theta\left(2^{2m}  e^{-2m} m^{2m+\frac{d}{2}}\right),
$$
and in particular
\begin{equation}
\label{precise}
\left[\diri{1}{m}(\mathbb B)\right]^{\fr{1}{2m}}=\fr{2m}{e}+\frac{d}{2e}\log m +\bo(1),
\end{equation}
as $m\to\infty$. 
\end{cor}

\begin{proof}
To get the expansion~\eqref{precise} we combine the upper and lower bounds from Theorem~\ref{thmA} together with the expansion~\eqref{stirlingnist}. In particular,
\[
\begin{array}{lll}
\log\left(\fr{2^{2m}\Gamma^2(m+1)\Gamma(2m+d/2+1)}{(m+d/2)\Gamma(d/2)\Gamma(2m+1)}\right) & = & 2m\log2+(2m+1)\log m-2m +\log(2\pi)\eqskip
& & \hspace*{5mm}  +(2m+(d+1)/2)\log(2m)-2m+\log\sqrt{2\pi} \eqskip
& & \hspace*{10mm}-\log m-\log\Gamma(d/2)-(2m+1/2)\log(2m)\eqskip
& & \hspace*{15mm} +2m-\log\sqrt{2\pi}+\bo(m^{-1})\eqskip
& = & 2m\log\fr{2m}{e} +\fr{d}{2}\log m+\log\fr{2^{d/2}2\pi}{\Gamma(d/2)}+\bo(m^{-1}),
\end{array}
\]
and
\[
\begin{array}{lll}
\log\left(\fr{2^{2m}\Gamma(m+1)\Gamma(m+d/2)}{\Gamma(d/2)}\right) & = & 2m\log2+(m+1/2)\log m-m+\log\sqrt{2\pi}\eqskip
& & \hspace*{5mm}  +(m+(d-1)/2)\log m-m+\log\sqrt{2\pi}+\bo(m^{-1})\eqskip
& = & 2m\log\fr{2m}{e} +\fr{d}{2}\log m+\log\fr{2\pi}{\Gamma(d/2)}+\bo(m^{-1}),
\end{array}
\]
from which we get
\[
\begin{array}{lll}
\left[\diri{1}{m}(\mathbb B)\right]^{\fr{1}{2m}} & = & \exp\left(\log\fr{2m}{e} +\fr{d}{4m}\log m+\bo(m^{-1})\right)\eqskip
& = & \fr{2m}{e}\left(1+\fr{d}{4m}\log m+\bo\left(\fr{(\log m)^2}{m^2}\right)\right)\left(1+\bo(m^{-1})\right)\eqskip
& = & \fr{2m}{e}+\fr{d}{2e}\log m +\bo(1).  
\end{array}
\]
\end{proof}

By combining the upper bound from Theorem~\ref{ballupperb} and the lower bound in Theorem~\ref{eineqaux}, it is also possible to obtain an estimate in the more general case
of $\diri{1}{m,m-h}(\mathbb B)$ where $h$ is fixed and $m\to\infty$.

\begin{thm}\label{ballasympt}
For any $m,h\in\mathbb N$ with $h<m$ we have
$$
\diri{1}{m-h}(\mathbb B)\le\diri{1}{m,m-h}(\mathbb B) = \bo\left(2^{2m}e^{-2m}m^{2m-h+\frac{d}{2}}\right),
$$
and similarly,
$$
\left[\diri{1}{m,m-h}(\mathbb B)\right]^{\fr{1}{2(m-h)}}=\fr{2m}{e}+\bo(\log m),
$$
as $m\to\infty$ (and fixed $h$).
\end{thm}

\begin{proof}
 The inequality
$$
\diri{1}{m-h}(\mathbb B)\le\diri{1}{m,m-h}(\mathbb B)
$$
comes directly from Theorem \ref{sameorder}. We observe that $\diri{1}{m-h}(\mathbb B)=\Theta\left(2^{2m}e^{-2m}m^{2m-2h+\frac{d}{2}}\right)$ and
$$
\left(\diri{1}{m-h}(\mathbb B)\right)^{1/(2m-2h)}=\frac{2m}{e}+\frac{d}{4e}\log m+\bo (1).
$$

On the other hand, from Theorem~\ref{ballupperb}
$$
\begin{array}{lll}
\left[\diri{1}{m,m-h}(\mathbb B)\right]^{\frac{1}{2(m-h)}} & \le & \left(\fr{2^{2m-2h}\Gamma^2(m-h+1)\Gamma(2m-h+1+d/2)}{(m+d/2)\Gamma(h+d/2)\Gamma(2m-2h+1)}\right)^{\frac{1}{2(m-h)}}\eqskip
& = & \fr{2m}{e}+\fr{2h+d}{4e}\log m+\bo(1).
\end{array}
$$
\end{proof}

\begin{remark}
Should the behaviour of the first eigenvalue of problem~\eqref{auxbuck} be as suggested in Remark~\ref{krarg}, namely, that
the corresponding eigenfunction is of one sign also when $m>t$, then Lemma~\ref{fund} would also provide a lower bound in this
case. Since this has the same asymptotic behaviour as the upper bound in the above theorem, namely, $\sigma=\Theta\left(2^{2m}e^{-2m}m^{2m-h+\frac{d}{2}}\right)$,
this would then give the precise behaviour for $\diri{1}{m,m-h}(\mathbb B)$.
\end{remark}


\subsection{Other asymptotic estimates: general domains and hyperrectangles}

Using the above results together with the rescaling property \eqref{scaling}, we may now discuss the case of a general domain, as well as the behaviour of eigenvalues
other than the fundamental tone. To do this, we shall use an inclusion argument which is basically a variant of that used by P\'{o}lya in~\cite{polya}.
More precisely, given a bounded domain $\Omega$, it is always possible to enclose it by a disjoint union of $k$ identical balls
$\overline\Omega={\ds \bigsqcup_{p=1}^k} \varrho_e \mathbb B$, where the parameter $\varrho_e$ can be chosen to be half the diameter of $\Omega$. In fact, it is possible
to consider the circumscribed ball to $\Omega$ together with $(k-1)$ disjoint copies. This then implies
$$
\lambda_k^{(m,m-h)}(\Omega)\ge \lambda_k^{(m,m-h)}(\overline\Omega)=\varrho_e^{-2(m-h)}\lambda_1^{(m,m-h)}(\mathbb B).
$$
On the other hand, it is also possible to find $k$ identical balls $\underline\Omega={\ds \bigsqcup_{p=1}^k} \varrho_i \mathbb B$ enclosed in $\Omega$ (the parameter
$\varrho_i$ can be taken to be the $k$-th part of the inradius). This implies
$$
\lambda_k^{(m,m-h)}(\Omega)\le \lambda_k^{(m,m-h)}(\underline\Omega)=\varrho_i^{-2(m-h)}\lambda_1^{(m,m-h)}(\mathbb B),
$$
proving Theorem~\ref{maingeral}.


On the other hand, the scaling property~\eqref{scaling} together with Theorem~\ref{ballasympt} immediately yield the following asymptotic expansion
$$
\left[\diri{1}{m,m-h}(\varrho\mathbb B)\right]^{\fr{1}{2(m-h)}}=\fr{2m}{e\varrho }+\bo(\log m),
$$
where $\varrho\mathbb B$ is the ball of radius $\varrho$ centred at the origin. For this reason, and also influenced by the scaling in the Weyl asymptotics~\eqref{weylmt},
it would be natural to conjecture that (at least) the first term in the order asymptotic would depend on the volume $|\Omega|$. However, as was already pointed out in
the~Introduction, we see that this is not the case. More precisely, we show that in the case of hyperrectangles the only geometric quantity to appear in this first term
is related to its inradius.

\begin{thm}\label{hrect}
Let $R\subseteq\mathbb R^d$ be the hyperrectangle
$
R={\ds \prod_{p=1}^d}(-a_p,a_p),
$
where $\min_p a_p=\ab$. Then
$$
\left[\lambda_1^{(m)}(R)\right]^{\frac{1}{2m}}=\frac{2m}{e\ab }+\bo(\log m),
$$
for $m\to\infty$.
\end{thm}

\begin{proof}
We first note that, since $\ab\mathbb B\subseteq R$, from the Rayleigh quotient we have
$
\lambda_1^{(m)}(R)\le\lambda_1^{(m)}(\ab \mathbb B),
$
while for a lower bound we make careful use of the one-dimensional bounds. From~\eqref{lb-est} we have
\begin{equation*}
\int_{-1}^1u^{(m)}(x)^2dx\ge (2m)!\int_{-1}^1u(x)^2dx,
\end{equation*}
for any function $u\in C^m_0(-1,1)$, implying that
$$
\int_{-a_p}^{a_p}u^{(m)}(x)^2dx\ge (2m)! a_p^{-2m}\int_{-a_p}^{a_p} u(x)^2dx,
$$
for any function $u\in C^m_0(-a_p,a_p)$. Therefore, applying Fubini's theorem we obtain
\begin{equation*}
\int_{R}\partial^\alpha u(x)^2dx\ge (2\alpha)!a^{-2\alpha}\int_{R}u(x)^2dx,
\end{equation*}
for any function $u\in C^m_0(R)$, and for any multi-index $\alpha$ such that $|\alpha|\le m$, where we have used
the multi-index notation $a^{-2\alpha}=\prod_p (a_p)^{-2\alpha_p}$. Now, integrating by parts we obtain
$$
\int_R|\nabla^m u|^2=\int_R|D^m u|^2,
$$
where $D^m u$ represents the tensor of derivatives of order $m$, and in particular
$
|D^mu|^2=\dsum_{|\alpha|=m}\frac{m!}{\alpha!}|\partial^\alpha u|^2.
$
Therefore
\begin{equation*}
\fr{\int_{R}|D^m u(x)|^2dx}{\int_{R}|u(x)|^2dx}
\ge \sum_{|\alpha|=m}\frac{m!(2\alpha)!}{\alpha!}a^{-2\alpha}\ge (2m)!\sum_{p=1}^d\frac 1{a_p^{2m}},
\end{equation*}
for any $u\in C^{m}_0(R)$, so taking the infimum we obtain
$$
\lambda_1^{(m)}(R)\ge (2m)!\sum_{p=1}^d\frac 1{a_p^{2m}}\ge \frac{(2m)!}{\ab^{2m}}.
$$
Since
$
\left[(2m)!\right]^{\fr{1}{2m}}=\fr{2m}{e}+\fr 1 e \log m+\bo(1),
$
this implies
$
\left[\lambda_1^{(m)}(R)\right]^{\fr{1}{2m}}=\fr{2m}{e\ab}+\bo(\log m).
$
\end{proof}


\section{On an inequality of Payne's}\label{payneineq}

We now present a generalization of an inequality due to Payne.

\begin{thm}
For any domain $\Omega$ in $\mathbb R^d$ with finite measure, for any $m,t \in\mathbb N$ one has
\begin{equation}
\label{truegenpayne}
\diri{2}{m,t}(\Omega)\le\diri{1}{m+1,t}(\Omega).
\end{equation}
\end{thm}

\begin{remark}
We remark that Payne proved inequality~\eqref{truegenpayne} in the case $m=t=1$~\cite[inequality~(45)]{payne}, however his proof contained a gap later filled in by
Friedlander~\cite{fried04}. We adapt here those arguments.
\end{remark}

\begin{proof}
Let us denote by $w_m$ an eigenfunction associated with $\diri{1}{m,t}(\Omega)$ and by $w_{m+1}$ an eigenfunction associated with $\diri{1}{m+1,t}(\Omega)$. We distinguish two cases.

{\bf Case $\int_\Omega \nabla^{m-t}w_m\cdot \nabla^{m-t}w_{m+1}\neq 0$.} In this case we may define the coefficients
$$\begin{array}{ll}
a_j=-\fr{\int_\Omega \nabla^{m-t}w_m\cdot \nabla^{m-t}\partial_j w_{m+1}}{\int_\Omega \nabla^{m-t}w_m\cdot \nabla^{m-t}w_{m+1}}, 
 & j=1,\dots,d,
 \end{array}
$$
where $\partial_j$ denotes the $j$-th partial derivative. With this definition, the function
$
\psi_j=a_j w_{m+1}+\partial_j w_{m+1}
$
is orthogonal to $w_m$ in the scalar product of $H^m_0(\Omega)$, making it a viable test function for the Rayleigh quotient of $\diri{2}{m,t}(\Omega)$ and in particular
$$
\diri{2}{m,t}(\Omega)\le \fr{\int_\Omega |\nabla^{m}\psi_j|^2}{\int_\Omega |\nabla^{m-t}\psi_j|^2}=\fr{a_j^2\int_\Omega |\nabla^{m}w_{m+1}|^2
+\int_\Omega |\nabla^{m}\partial_jw_{m+1}|^2}{a_j^2\int_\Omega |\nabla^{m-t}w_{m+1}|^2+\int_\Omega |\nabla^{m-t}\partial_jw_{m+1}|^2},
$$
where we used the fact that
$$
\int_\Omega \nabla^{p}w_{m+1}\cdot\nabla^p\partial_j w_{m+1}=\fr{1}{2}\int_\Omega\partial_j|\nabla^p w_{m+1}|^2=0
$$
for any $p\le m$. Now, since
$$
\alpha\le \fr{p_j}{q_j}\ \Rightarrow\ q_j\alpha\le p_j\ \Rightarrow \alpha\sum q_j\le\sum p_j,
$$
we deduce that
$$
\diri{2}{m,t}(\Omega)\le \fr{\left(\sum_ja_j^2\right)\int_\Omega |\nabla^{m}w_{m+1}|^2+\int_\Omega |\nabla^{m+1}w_{m+1}|^2}{\left(\sum_ja_j^2\right)\int_\Omega |\nabla^{m-t}w_{m+1}|^2+\int_\Omega |\nabla^{m+1-t}w_{m+1}|^2},
$$
and from~\eqref{firststepsame1} we get
$$
\diri{2}{m,t}(\Omega)\le \fr{\left(\sum_ja_j^2\right)\fr{\int_\Omega |\nabla^{m+1}w_{m+1}|^2}{\int_\Omega |\nabla^{m+1-t}w_{m+1}|^2}\int_\Omega |\nabla^{m-t}w_{m+1}|^2+\int_\Omega |\nabla^{m+1}w_{m+1}|^2}{\left(\sum_ja_j^2\right)\int_\Omega |\nabla^{m-t}w_{m+1}|^2+\int_\Omega |\nabla^{m+1-t}w_{m+1}|^2}.
$$
Using the fact that $\int_\Omega |\nabla^{m+1}w_{m+1}|^2=\diri{1}{m+1,t} \int_\Omega |\nabla^{m+1-t}w_{m+1}|^2$ allows us to conclude.

{\bf Case $\int_\Omega \nabla^{m-t}w_m\cdot \nabla^{m-t}w_{m+1}= 0$.} In this case we know that $w_m$ and $w_{m+1}$ must be linearly independent, so we can consider the 2-dimensional space they generate within $H^m_0(\Omega)$ as a test space for $\diri{2}{m,t}$. In particular, any function $\psi$ in such space will be of the form $a w_m+bw_{m+1}$ for some $a,b\in\mathbb R$. In addition we also have that 
$$
\int_\Omega \nabla^{m}w_m\cdot \nabla^{m}w_{m+1}=\diri{1}{m,t}\int_\Omega \nabla^{m-t}w_m\cdot \nabla^{m-t}w_{m+1}=0.
$$
Now we can compute
\[
\begin{array}{lll}
\int_\Omega |\nabla^m(a w_m+bw_{m+1})|^2
& = & a^2\int_\Omega |\nabla^m w_m|^2+b^2\int_\Omega |\nabla^m w_{m+1}|^2+2ab\int_\Omega \nabla^m w_m\cdot\nabla^m w_{m+1}\eqskip
& = & a^2\diri{1}{m,t}\int_\Omega |\nabla^{m-t} w_m|^2+b^2\int_\Omega |\nabla^m w_{m+1}|^2\eqskip
& \le &  a^2\diri{1}{m+1,t}\int_\Omega |\nabla^{m-t} w_m|^2+b^2\diri{1}{m+1,t}\int_\Omega |\nabla^{m-t} w_{m+1}|^2\eqskip
& = & \diri{1}{m+1,t}\int_\Omega |\nabla^{m-t}(a w_m+bw_{m+1})|^2,
\end{array}
\]
yielding
$$
\diri{2}{m,t}\le\fr{\int_\Omega |\nabla^m(a w_m+bw_{m+1})|^2}{\int_\Omega |\nabla^{m-t}(a w_m+bw_{m+1})|^2}\le \diri{1}{m+1,t}.
$$
\end{proof}

The considerations in Appendix~\ref{ballsection} and inequality~\eqref{truegenpayne} in the original version comparing the Dirichlet Laplacian eigenvalues and
those of the buckling problem led to the conjecture that the same shift-type inequality should be valid for all eigenvalues (see \cite{friedp,liup}). We state
it here in our general context.

\begin{conj*}[Generalised Payne buckling inequality]
For any domain $\Omega$ in $\mathbb R^d$ with finite measure, for any $k,m,t \in\mathbb N$ one has
\begin{equation}
\label{dtntrue}
\diri{k+1}{m,t}(\Omega)\le\diri{k}{m+1,t}(\Omega).
\end{equation}
\end{conj*}

We observe that Corollary \ref{buckevs} strongly hints towards the validity of inequality \eqref{dtntrue} whenever $\Omega$ is a ball. In fact, even though the inequality
itself cannot be improved any further in this case, we know that when $m$ grows the operator ``looses'' all the eigenvalues coming from
$
\det L(t,h)=0,
$
for $h<m-t$. Notice also that, if the interlacing of zeroes similar to those of the Bessel functions were proved, the validity of inequality \eqref{dtntrue} for balls would be an
immediate consequence, and in particular we deduce that it holds at least in the case $t=1,2$.

However, for domains other than balls a proof of \eqref{dtntrue} seems out of reach at the moment. We note that Payne's argument \cite{payne} in the form used to
prove~\eqref{truegenpayne} does not apply for $k>1$. On the other hand, Friedlander's correction \cite{fried04} works for every $k$ but only under an orthogonality condition that
appears to be substantial.

Another possible approach is the use of a Dirichlet-to-Neumann argument. While such an argument is very elegant, it relies on the use of specific test functions that, however,
seem extremely difficult to individuate in this situation. This approach was used by Liu \cite{liup}, but unfortunately his proof contained a gap in the use of
test functions (see also \cite{friedp}). We also tried a Dirichlet-to-Neumann approach but considering intermediate (Navier) operators, which again turned out
to be problematic concerning the use of test functions. While we think that it may be possible to construct such test functions relying on a good use of layer potentials and/or Green
kernel representations, we however suspect that the conjecture is false for a general domain $\Omega$. To this end, we mention that in~\cite{blps} it was shown that, in the
one-dimensional case, the original conjecture $\diri{k+1}{1,1}\le\diri{k}{2,1}$ is false. It is not clear though how this relates to the higher--dimensional case. For completeness,
we show here that~\eqref{dtntrue} is false in the one-dimensional setting for $t=1$ and any choice of $m\ge 1$.

To this end, we recall that problem \eqref{dirichletbuckling} in dimension one reads
\begin{equation}
\label{dirbuc1}
\begin{cases}
(-1)^m u^{(2m)}=\diri{}{m,t} (-1)^{m-t}u^{(2m-2t)}, & \text{for }x\in(-1,1),\eqskip
u^{(p)}(\pm 1)=0, & \text{for }p=0,1,\dots,m-1.
\end{cases}
\end{equation}

A differential equation of the form~\eqref{dirbuc1} is usually solved through the associated characteristic equation, which in this context is still possible
but provides little insight. Alternatively, building upon the concept that it is always possible to separate even and odd eigenfunctions, we see that solutions
of~\eqref{dirbuc1} can be analysed as those of the higher-dimensional ball, the crucial difference being spherical harmonic functions that do not make much
sense for $d=1$. Nevertheless, it is possible to interpret spherical harmonics on $\mathbb S^0=\{\pm1\}$ as even and odd functions (to reflect the dichotomy present
in the eigenfunctions) which, in contrast to what happens in higher dimensions, here then only translate into two families of eigenfunctions: even and odd ones. Even
eigenfunctions are written in terms of $\tilde J_0=J_{-1/2}$, odd ones in terms of $\tilde J_1=J_{1/2}$. In this way, it becomes possible to write somewhat explicitly all
the eigenfunctions in the same fashion as for the higher dimension case. In particular, this argument allows us to prove the following result.

\begin{thm}
\label{pineq1d}
Let $t=1$, $m\in\mathbb N$. The eigenvalues of problem~\eqref{dirbuc1} are all simple and are given by $\left\{j_{m\pm\fr12,k}^2\right\}_{k\in\mathbb N}$. In particular, for all $m$ there is
a countable set of eigenvalues of problem \eqref{dirbuc1} which coincide with a countable set of eigenvalues of problem~\eqref{dirbuc1} where $m$ is replaced by $m+1$.
In addition,
$$
\diri{k}{m,1}(-1,1)\le \diri{k}{m+1,1}(-1,1)\le\diri{k+1}{m,1}(-1,1),\quad\forall m\in\mathbb N.
$$
\end{thm}
\begin{proof}
The first part follows directly from the discussion above. The inequality is an immediate consequence of the interlacing of Bessel zeroes (see~\cite[Section 10.21(i)]{nist}).
\end{proof}

We remark that Theorem~\ref{pineq1d} for the specific case $m=1$ may already be found in~\cite[Section 4]{blps}, where problem~\eqref{dirbuc1} was instead set in $[0,1]$
and where the corresponding eigenvalues are completely characterized in terms of trigonometric equations. In any case, we see one deep difference that stands out, related to
multiplicities. More precisely, while on a segment all eigenvalues are still simple, this is not necessarily the case in higher-dimensional balls. This difference
in behaviour between one and higher dimensions may be understood in terms of this multiplicity difference and, in fact, inequality~\eqref{dtntrue} does not hold for 
even values of $k$, while for odd $k$ it becomes an equality.

\appendix
\section{Eigenfunctions and eigenvalues on the ball}
\label{ballsection}

In this appendix we consider problem \eqref{dirichletbuckling} on a ball and derive expressions for the eigenfunctions and the corresponding eigenvalues. The strategy we follow is
the standard decomposition of eigenfunctions into a radial part times a spherical part, as is well known for the Dirichlet Laplacian ($m=t=1$), the Dirichlet
bilaplacian ($m=t=2$), and the classical buckling problem ($m=2,t=1$); cf.\ \cite{ashben, chas1}.

We set $\mathbb B$ to be the unit ball in $\mathbb R^d$ centred at the origin, and use the spherical coordinates $(r,\theta)\in \mathbb R_+\times \mathbb S^{d-1}$.
In this case, arguing as in \cite[Section 6]{chas1}, we arrive to the conclusion that any eigenfunction of problem~\eqref{dirichletbuckling} can be written as the
product of a radial function times a spherical harmonic function. For the reader's convenience we recall the argument here. 

First of all, we observe that the operator $P$ associated with problem \eqref{dirichletbuckling} is rotation invariant, hence it commutes with the Laplace-Beltrami
operator $\Delta_{\mathbb S^{d-1}}$. In addition, since it is self-adjoint and with compact resolvent, each eigenspace $E$ is finite dimensional and is mapped into itself
by $\Delta_{\mathbb S^{d-1}}$ due to the commutativity. In particular, since $P$ and $\Delta_{\mathbb S^{d-1}}$ commute on the finite dimensional space $E$, it is easy to
see that they must be simultaneously diagonalizable on $E$ (e.g., thinking of the operators as matrices acting on finite dimensional linear spaces). The eigenfunctions of $\Delta_{\mathbb S^{d-1}}$ on
$\partial\mathbb B$ are the spherical harmonic functions, and therefore on $\mathbb B$ the eigenfunctions take the form of the product of a radial function times a spherical
harmonic function. This means that also the eigenfunctions of problem \eqref{dirichletbuckling} must take this form.

We observe that, while it is customary to assume this decomposition as an ansatz for any rotation invariant eigenvalue problem, the sole ansatz would require to prove
the completeness of the eigenfunctions a posteriori (see e.g., \cite[Chapter II]{chavel}). Here we have shown that all eigenfunctions must adhere to such a decomposition, so we already know a priori that they
form a complete system for $H^m_0(\mathbb B)$ since the operator $P$ is self-adjoint with compact resolvent.

At this point, if we rewrite the equation of problem \eqref{dirichletbuckling} as
\begin{equation}
\label{formaequaz}
(-\Delta)^{m-t}((-\Delta)^t -\lambda^{(m,t)})u=0,
\end{equation}
we obtain the following form for the eigenfunctions
\begin{equation}
\label{solutions}
\left[\sum_{p=1}^{m-t}C_p r^{\ell+2(p-1)}+\sum_{p=1}^t C_{p+m-t} \tilde J_{\ell}(\gamma_p r)\right]S_\ell(\theta),
\end{equation}
for some $\ell\in\mathbb N_0$, with the convention that if $m-t=0$ then the first sum in the square brackets disappears. Here $\tilde J_\ell(z)=z^{1-d/2}J_{\ell+d/2-1}(z)$ (for $z\in\mathbb C$),
where $J_\kappa$ is the classical Bessel function of the first kind, $\gamma_p=\left(\lambda^{(m,t)}\right)^{1/2t}e^{\fr{i(p-1)\pi}{t}}$ for any $p=1,\dots,t$, and $S_\ell$ is a spherical
harmonic of degree $\ell$.

Solutions of equation \eqref{formaequaz} can be written in the form \eqref{solutions} by virtue of the fact that the operator $((-\Delta)^t -\lambda^{(m,t)})$ can be decomposed as
$
(-\Delta-\gamma_1^2)\cdots(-\Delta-\gamma_t^2),
$
where all the involved operators commute, and moreover they also commute with $(-\Delta)^{m-t}$. Now, by virtue of the commutativity with
$\Delta_{\mathbb S^{d-1}}$,  and since solutions must be smooth at the origin, solutions of $(-\Delta-\gamma_p^2)V=0$ must take the form 
$\tilde J_\ell(\gamma_p r) S_\ell(\theta),$ 
for $\ell\in\mathbb N_{0}$, while the Almansi decomposition (cf.\ \cite{almansi}; see also \cite{bfgm} for a recent reference) states that solutions of $(-\Delta)^{m-t}V=0$ belong to the linear space generated by
$$
\left\{r^{\ell+2p}S_\ell(\theta):\ \ S_\ell \text{ any spherical harmonic of degree }\ell\in\mathbb N_{0},\ p=0,\dots, m-t-1\right\}.
$$
Finally, when combining these (partial) solutions together into \eqref{solutions}, the boundary conditions in \eqref{dirichletbuckling} force all their spherical
parts to coincide. Moreover, even though in principle functions of the form~\eqref{solutions} may, in general, be complex-valued, since the original problem is scalar
with real coefficients, and self-adjoint, it is possible to choose the coefficients $C_{k}$ so as to have real eigenfunctions.

Imposing now the boundary conditions on \eqref{solutions} yields, as usual, a system of equations for the coefficients $C=(C_1,\dots,C_m)$ of the type
$
A\cdot C=0,
$
and the matrix $A$ has to have a zero determinant for the system to admit nontrivial solutions. In particular, the condition $\det A=0$ becomes the defining equation for the eigenvalues $\lambda^{(m,t)}$. On the other hand, we will not actively look for the vector of coefficients $C$, but focus on the eigenvalues only. We remark though that the solutions $C$ allow to write all the eigenfunctions, which form an eigenbasis for $L^2(\Omega)$ (eventually after normalization).

The matrix $A$ will depend on $m$, $t$, and $\ell$ ($A=A(m,t,\ell)$)
and its elements $a_{ij}(m,t,\ell)$ can be defined as follows 
$$
a_{ij}(m,t,\ell)=
\left\{
\begin{array}{ll}
\dprod_{q=0}^{i-2}(\ell+2(j-1)-q), & 1\le i\le m, 1\le j\le m-t,\eqskip
\gamma_{j-m+t}^{i-1}\tilde J_\ell^{(i-1)}(\gamma_{j-m+t}), &  1\le i\le m, m-t< j\le m,
\end{array}
\right.
$$
where we used the convention that $\dprod_{q=0}^{-1}(\ell+2(j-1)-q)=1$. 


In order to study the determinant of $A$, it is convenient to consider the matrix $B(n,\ell)=(b_{ij}(n,\ell))_{1\le i,j\le n}$ with coefficients defined by
$$
b_{ij}(n,\ell)=\dprod_{q=0}^{i-2}(\ell+2(j-1)-q),
$$
which we observe is the $n\times n$-minor of $A(m,t,\ell)$ for $n\le m-t$. Similarly, we consider the matrix ($p\in\mathbb N$, $k\in\mathbb N_0$)
$$
L(p,k)=\left((-\gamma_{j})^{i-1}\tilde J_{k+i-1}(\gamma_{j})\right)_{1\le i,j \le p}.
$$
We have the following

\begin{thm}
\label{determinanti}
For any $m,t\in\mathbb N$ with $t\leq m$ and $\ell\in\mathbb{N}_{0}$, we have
$$
\det A(m,t,\ell)=(-1)^{t(m-t)}\det B(m-t,\ell) \det L(t,\ell+m-t)\prod_{j=1}^t(\gamma_j^{m-t}).
$$
Moreover, for any $n\in\mathbb N_0$
$$
\det B(n,\ell)=\dprod_{q=0}^{n-1}(q!2^q).
$$
\end{thm}
\begin{remark}
We remark that under our convention $\det B(0,\ell)=1$, even though $B(0,\ell)$ is not well defined as a matrix.
\end{remark}
\begin{proof}
Applying Lemma \ref{reduction} below it is possible to reduce $A(m,t,\ell)$ to a $2\times2$-block matrix, where the two blocks on the diagonal are precisely $B(m-t,\ell)$ and
$$
\left((-\gamma_{j})^{m-t+i-1}\tilde J_{\ell+m-t+i-1}(\gamma_{j})\right)_{1\le i,j \le t},
$$
while the block under the diagonal is identically zero.

As for the determinant of $B(n,\ell)$, the reduction leads to a triangular matrix whose diagonal terms are precisely $((i-1)!2^{i-1})$.
\end{proof}

\begin{cor}
\label{buckevs}
All the eigenvalues of problem \eqref{dirichletbuckling} in $\mathbb B$ are solutions to the equation
\begin{equation}
\label{genbess}
\det L(t,\ell+m-t)=0,
\end{equation}
for some $\ell\in\mathbb N_0$. 
\end{cor}

Let us remark that equation \eqref{genbess} trivially admits as solutions zeroes of Bessel functions whenever $t=1$, as expected (this is the case for instance of both the Laplacian
and the buckling problem). In this case it is then immediate to deduce that there exist infinitely many solutions for any $\ell\in\mathbb N_0$. Similarly, when $t=2$ the equation reduces
to one of the type
$$
J_{\ell+p}(\gamma)I_{\ell+p+1}(\gamma)+J_{\ell+p+1}(\gamma)I_{\ell+p}(\gamma)=0,
$$
where $I_\kappa$ is the modified Bessel function of the first kind. Here again it is possible to show that there are infinitely many solutions for each $\ell$ using the interlacing properties of the zeroes of Bessel functions $J$. However, for $t\ge 3$
the situation becomes more involved, and although it is not clear a priori that solutions exist for each $\ell$, we do know that the set of all solutions for all $\ell$ must be countable.
More generally, the question of existence and interlacing of the zeroes of equation~\eqref{genbess} is as interesting as it is difficult, and raises the possibility of a generalized
version of Bourget's hypothesis (see e.g., in~\cite{ker}).

We conclude this section with a technical lemma needed for Theorem \ref{determinanti}. In order to properly state it we need to define the quantity $\alpha(p,i,k)$ for $p\in\mathbb N$, $k\in\mathbb N_0$,
and $i\in\mathbb N_{-1}$ with $i<p$ as follows:
\begin{equation*}
\begin{cases}
\alpha(p,-1,k)=0, & \forall_{p,k},\eqskip
\alpha(1,0,k)=k, & \forall_{k},\eqskip
\alpha(p+1,i,k)=\alpha(p,i-1,k)-(2p+k-i)\alpha(p,i,k), & {\rm for\ } 0\le i\le p-1, \forall_k,\eqskip
\alpha(p+1,p,k)=\alpha(p,p-1,k)+p+k, & \forall_k. 
\end{cases}
\end{equation*}

\begin{lem}
\label{reduction}
For any $k\in\mathbb N_0$ and $m\in\mathbb N$, the following identities hold true.
\begin{enumerate}[i)]
\item $\dprod_{q=0}^{m-1}(k+2(j-1)-q)=\sum_{i=1}^{m}\left[\alpha(m,i-1,k)\dprod_{q=0}^{i-2}(k+2(j-1)-q)\right]$, for any $1\le j\le m$.
\item $\dprod_{q=0}^{m-1}(k+2m-q)=m!2^{m}+\sum_{i=1}^{m}\left[\alpha(m,i-1,k)\dprod_{q=0}^{i-2}(k+2m-q)\right]$.
\item $z^m \tilde J_k^{(m)}(z)=(-z)^m\tilde J_{k+m}(z)+\dsum_{i=1}^{m}\alpha(m,i-1,k)z^{i-1}\tilde J_k^{(i-1)}(z)$, for any $z\in\mathbb C\setminus\{0\}$.
\end{enumerate}
\end{lem}

\begin{proof}
We start with identity~{\it i)}. We fix $k$ and we prove the identity by induction on $m$. We observe that the case $m=1$ is trivial, so we suppose that the
equality holds for $m$ substituted by $m-1$ and for any $1\le j\le m-1$. From the definition of the coefficients $\alpha$ we obtain
\begin{equation}
 \begin{array}{lll}
\label{primeiro}
\dsum_{i=1}^{m}\left[\alpha(m,i-1,k)\dprod_{q=0}^{i-2}(k+2(j-1)-q)\right] \eqskip
\hspace*{5mm}= \dsum_{i=1}^{m-1}\left[[\alpha(m-1,i-2,k)-(2m-1+k-i)\alpha(m-1,i-1,k)]\dprod_{q=0}^{i-2}(k+2(j-1)-q)\right]\eqskip
\hspace*{10mm}+\alpha(m-1,m-2,k)\dprod_{q=0}^{m-2}(k+2(j-1)-q)+(k+m-1)\dprod_{q=0}^{m-2}(k+2(j-1)-q)\eqskip
\hspace*{5mm}=  \dsum_{i=1}^{m-1}\left[\alpha(m-1,i-1,k)\dprod_{q=0}^{i-1}(k+2(j-1)-q)\right]\eqskip
 \hspace*{10mm}-\dsum_{i=1}^{m-1}\left[(2m-1+k-i)\alpha(m-1,i-1,k)\dprod_{q=0}^{i-2}(k+2(j-1)-q)\right]\eqskip
 \hspace*{15mm}+(k+m-1)\dprod_{q=0}^{m-2}(k+2(j-1)-q)\eqskip
\hspace*{5mm}= (2(j-1)-2m+2)\dsum_{i=1}^{m-1}\left[\alpha(m-1,i-1,k)\dprod_{q=0}^{i-2}(k+2(j-1)-q)\right]\eqskip
\hspace*{10mm} +(k+m-1)\dprod_{q=0}^{m-2}(k+2(j-1)-q).
 \end{array}
\end{equation}
Now, if $j\le m-1$, by the inductive hypothesis \eqref{primeiro} equals
$$
(k+2(j-1)-m+1)\dprod_{q=0}^{m-2}(k+2(j-1)-q),
$$ 
while if $j=m$ the first term in the last equality vanishes. In both cases, the claim is proved.

We now turn to {\it ii)}. We first observe that, following the same strategy as in~\eqref{primeiro}, for any $0\le\xi\le m-1$,
\begin{equation}
\label{segundo1}
\begin{array}{lll}
\dsum_{i=1}^{m-\xi}\left[\alpha(m-\xi,i-1,k)\dprod_{q=0}^{i-2}(k+2m-q)\right] & = & 2(\xi+1)\dsum_{i=1}^{m-\xi-1}\left[\alpha(m-\xi-1,i-1,k)\dprod_{q=0}^{i-2}(k+2m-q)\right]\eqskip
& & \hspace*{5mm} +(k+m-\xi-1)\dprod_{q=0}^{m-\xi-2}(k+2m-q).
\end{array}
\end{equation}
In particular, starting from
\begin{equation}
\label{segundo2}
\begin{array}{lll}
\dsum_{i=1}^{m}\left[\alpha(m,i-1,k)\dprod_{q=0}^{i-2}(k+2m-q)\right]
& = & 2\dsum_{i=1}^{m-1}\left[\alpha(m-1,i-1,k)\dprod_{q=0}^{i-2}(k+2m-q)\right]\eqskip
& & \hspace*{5mm}+(k+m+1)\dprod_{q=0}^{m-2}(k+2m-q),
\end{array}
\end{equation}
and applying \eqref{segundo1} on \eqref{segundo2}, by induction, we obtain
\begin{equation*}
\sum_{i=1}^{m}\left[\alpha(m,i-1,k)\dprod_{q=0}^{i-2}(k+2m-q)\right]
=\sum_{h=0}^{m-1}\left[2^h(h!)(k+m-h-1)\dprod_{q=0}^{m-h-2}(k+2m-q)\right].
\end{equation*}
It is now easy to prove by induction that
\begin{gather*}
\sum_{h=0}^{\xi}\left[2^h(h!)(k+m-h-1)\dprod_{q=0}^{m-h-2}(k+2m-q)\right]
=\dprod_{q=0}^{m-1}(k+2m-q)-2^{\xi+1}(\xi+1)!\dprod_{q=0}^{m-\xi-2}(k+2m-q),
\end{gather*}
which is valid for $0\le\xi\le m-1$. Choosing now $\xi=m-1$ we obtain the desired equality.

We now turn to {\it iii)}. We recall that, by the properties of Bessel functions (see \cite[formula 10.6.2]{nist})
$$
\tilde J_\kappa'(z)=-\tilde J_{\kappa+1}(z)+\fr \kappa z \tilde J_\kappa(z),
$$
for any $\kappa\ge 0$, hence the identity is trivial for $m=1$. We now derive $\tilde J_k^{(m)}$ obtaining
\[
\tilde J_k^{(m+1)}(z)
=(-1)^m\tilde J_{k+m}'(z)
-\sum_{i=1}^{m}\fr{\alpha(m,i-1,k)(m-i+1)}{z^{m-i+2}}\tilde J_k^{(i-1)}
+\sum_{i=1}^{m}\fr{\alpha(m,i-1,k)}{z^{m-i+1}}\tilde J_k^{(i)}.
\]
Finally, from
\[
\begin{array}{lll}
\tilde J_{k+m}'(z) & = & -\tilde J_{k+m+1}+\fr{k+m}{ z} \tilde J_{k+m}(z)\eqskip
& = & -\tilde J_{k+m+1}+\fr{k+m}{ z}\left[(-1)^m\tilde J_k^{(m)}(z)+(-1)^{m+1}\sum_{i=1}^{m}\fr{\alpha(m,i-1,k)}{z^{m-i+1}}\tilde J_k^{(i-1)}(z)\right]
\end{array}
\]
we derive
\[
\begin{array}{lll}
\tilde J_k^{(m+1)}(z) & = & (-1)^{m+1}\tilde J_{k+m+1}(z)\eqskip
& & \hspace*{5mm}+\dsum_{i=1}^{m}\fr{\alpha(m,i-2,k)-(2m+k-i+1)\alpha(m,i-1,k)}{z^{m-i+2}}\tilde J_k^{(i-1)}(z)\eqskip
& & \hspace*{10mm} +\fr{\alpha(m,m-1,k)+k+m}{z}\tilde J_k^{(m)}(z),
\end{array}
\]
concluding the proof.
\end{proof}

\section{Solutions to the auxiliary problem~(\ref{auxbuck}) and related calculations}
\label{techcompt}

We collect here some technical computations that are used in Section~\ref{genasympt}. For the reader's convenience, we recall the definition of the double factorial which will be used extensively in what follows:
\begin{equation*}
n!!=\left\{\begin{array}{l}
(-1)!!=1,\eqskip
\fr{n!}{(n-1)!!}.
\end{array}\right.
\end{equation*}
In particular we note that the double factorial differs from the (classical) factorial as $n!!$ is the product from $n$ down but with step 2: $n!!=n(n-2)(n-4)\cdots$.
We will also use the equivalent expressions
\begin{equation*}
\begin{array}{lll}
n!!=
\left\{\begin{array}{ll}
k!2^{k},& \text{if\ } n=2k,\eqskip
\fr{(2k)!}{k!2^k},&\text{if\ } n=2k-1,
\end{array}\right.
&
\mbox{and}
&
n!!=
\left\{\begin{array}{ll}
2^{n/2}\Gamma(n/2+1),& \text{if\ } n=2k,\eqskip
2^{n/2}\Gamma(n/2+1)\sqrt{2/\pi},&\text{if\ } n=2k-1.
\end{array}\right.
\end{array}
\end{equation*}

We also recall the following equality which is obtained via a direct application of properties of Euler Gamma function
\begin{equation}
\label{calculus1}
\int_0^1(1-r^2)^\alpha r^\beta dr = \frac{\Gamma (\alpha +1) \Gamma \left(\frac{\beta +1}{2}\right)}{2 \Gamma \left(\alpha +\frac{\beta +3}{2}\right)}=\fr{(2\alpha)!!(\beta-1)!!}{(\beta+2\alpha+1)!!}=\fr{\alpha!}{2 \prod_{j=0}^{\alpha}\left( \fr{\beta+1}{2}+j\right)}.
\end{equation}

\begin{lem}\label{hyperlemma}
Let $f(x)=(1-|x|^2)^m$, and let $0\le t \le m$. Then
\begin{equation} \label{hypergeometric}
\int_{\mathbb B}|\nabla^t f|^2=2\pi\omega_{d-1}\frac{2^{2t-1}\Gamma^2(m+1)\Gamma(t+d/2)\Gamma(2m-2t+1)}{\Gamma^2(m-t+1)\Gamma(2m-t+1+d/2)}.
\end{equation}
\end{lem}

\begin{proof}
Setting $g(r)=(1-r^2)^m$, we have
$$
(-\Delta)^t f(x)=\left(-\fr 1 {r^{d-1}}\fr{\partial\ }{\partial r}\left(r^{d-1}\fr{\partial\ }
{\partial r}\right)\right)^t g(|x|),
$$
hence
$$
\int_{\mathbb B}|\nabla^t f|^2=\int_{\mathbb B}f(-\Delta)^t f=2\pi\omega_{d-1}\int_0^1g(r)\left(-\fr 1 {r^{d-1}}\fr{\partial\ }{\partial r}\left(r^{d-1}\fr{\partial\ }
{\partial r}\right)\right)^t g(r)r^{d-1}dr.
$$

We now observe that
\begin{equation}
\label{delta-f}
\left(\fr 1 {r^{d-1}}\fr{\partial\ }{\partial r}\left(r^{d-1}\fr{\partial\ }{\partial r}\right)\right)^t g(r)
=\sum_{h=t}^m\binom{m}{h}(-1)^{h}r^{2h-2t}\frac{(2h)!!}{(2h-2t)!!}\frac{(2h+d-2)!!}{(2h+d-2t-2)!!},
\end{equation}
from which we obtain
\begin{gather*}
\begin{array}{lll}
\dintd{0}{1}g(r)\left(-\fr 1 {r^{d-1}}\fr{\partial\ }{\partial r}\left(r^{d-1}\fr{\partial\ }{\partial r}\right)\right)^t g(r)r^{d-1}dr
& = & 2^{m+t}(m!)\dsum_{h=t}^m\binom{m}{h}(-1)^{h-t}\fr{h!}{(h-t)!}\frac{(2h+d-2)!!}{(2m+2h+d-2t-2)!!}\eqskip
& = & 2^{2t-1}\fr{(m!)^2}{(m-t)!}\dsum_{p=0}^{m-t}\binom{m-t}{p}\fr{(-1)^p}{\prod_{j=0}^{m-t}(t+p+d/2+j)}.
\end{array}
\end{gather*}
At this point we observe that, for any $p\ge 0$,
$$
\fr{1}{\prod_{j=0}^{m-t}(t+p+d/2+j)}=\fr{1}{\prod_{j=0}^{m-t}(t+d/2+j)}\frac{(t+d/2)_p}{(m+1+d/2)_p}=\fr{\Gamma(t+d/2)}{\Gamma(m+1+d/2)}\frac{(t+d/2)_p}{(m+1+d/2)_p},
$$
where $(a)_p$ is the usual Pochhammer symbol defined by $(a)_p={\ds \prod_{j=0}^{p-1}}(a+j)$. Now we use \cite[formula 15.2.4]{nist} to deduce that
$$
\sum_{p=0}^{m-t}\binom{m-t}{p}(-1)^{h-t}\fr{(-1)^p}{\prod_{j=0}^{m-t}(t+p+d/2+j)}=\frac{\Gamma(t+d/2)}{\Gamma(m+1+d/2)} F(-m+t,t+d/2;m+1+d/2;1),
$$
where $F(a,b;c;z)$ is the hypergeometric function, for which we recall~\cite[formula 15.4.20]{nist}
$$
F(a,b;c;1)=\fr{\Gamma(c)\Gamma(c-a-b)}{\Gamma(c-a)\Gamma(c-b)}\quad\text{if\ }c>a+b.
$$
Collecting all these formulas yields identity~\eqref{hypergeometric}.
\end{proof}

\begin{lem}
\label{fund}
Let $f(x)=(1-|x|^2)^m$. Then $f$ is an eigenfunction for~\eqref{auxbuck}, and its associated eigenvalue is
\begin{equation}\label{sigmaeigen}
\sigma=
\begin{cases}
\fr{(m+t)!!(m+t+d-2)!!}{(m-t)!!(m-t+d-2)!!}, & \text{ if } m-t \text{ is even},\eqskip
\fr{(m+t-1)!!(m+t+d-1)!!}{(m-t-1)!!(m-t+d-1)!!}, & \text{ if } m-t \text{ is odd}.
\end{cases}
\end{equation}
In particular, if $m=t$, then
$$
\sigma=
\fr{(2m)!!(2m+d-2)!!}{(d-2)!!}=2^{2m}\fr{\Gamma(m+1)\Gamma(m+d/2)}{\Gamma(d/2)},
$$
and for $d=1$,
$$
\sigma=
\begin{cases}
\fr{(m+t)!}{(m-t)!}, & \text{ if } m\neq t,\eqskip
(2m)!, & \text{ if } m=t.
\end{cases}
$$
\end{lem}

\begin{remark}
Denoting by $\left \lfloor{x}\right \rfloor $ the integer part of $x$, the value of $\sigma$ in~~\eqref{sigmaeigen} may be written as
$$
\sigma=2^{2t} \fr{\Gamma\left(\left\lfloor\frac{m+t+2}{2}\right\rfloor\right)\Gamma\left(\left\lfloor\frac{m+t+1}{2}\right\rfloor+\frac d 2\right)}{\Gamma\left(\left\lfloor\frac{m-t+2}{2}\right\rfloor\right)\Gamma\left(\left\lfloor\frac{m-t+1}{2}\right\rfloor+\frac d 2\right)},
$$
for all possible values of $m$ and $t$.
\end{remark}

\begin{proof}
From \eqref{delta-f} we already know that
$$
(-\Delta)^m f=\fr{(2m+d-2)!!(2m)!!}{(d-2)!!},
$$
so we are left with computing $(-1)^{m-t}\nabla^{m-t}\fr{1}{(1-|x|^2)^t}\nabla^{m-t}f$. We must distinguish two cases.

{\bf Case $m-t$ even.} Write $k=\fr{m-t}{2}$. We have to compute
$
(-1)^{m-t}\Delta_r^{k}\fr{1}{(1-r^2)^t}\Delta_r^{k}g,
$
where we have put $g(|x|)=f(x)$ and $\Delta_r=\left(\fr 1 {r^{d-1}}\fr{\partial\ }{\partial r}\left(r^{d-1}\fr{\partial\ }{\partial r}\right)\right)$.

We first show (by induction) that, if $m\ge 2k$, then
\begin{equation}
\label{evenoddcomp}
\Delta_r^{k}g(r) =(1-r^2)^{m-2k} P_k(r),
\end{equation}
where $P_k$ is an even polynomial of degree $2k$. Naturally the claim holds for $k=0$, so now we consider the claim valid for $k-1$. We have
\begin{equation*}
 \begin{array}{lll}
\Delta_r^{k}g(r) =\Delta_r(\Delta_r^{k-1}g(r)) & = & \Delta_r\left( (1-r^2)^{m-2k+2} P_{k-1}(r)\right)\eqskip
& = & (1-r^2)^{m-2k+2}\left(P_{k-1}''(r)+(d-1)r^{-1}P_{k-1}'(r)\right)\eqskip
& & \hspace*{5mm}-2(m+2k-2)(1-r^2)^{m-2k+1}\left(dP_{k-1}(r)+2rP_{k-1}'(r)\right)\eqskip
& & \hspace*{10mm} +4(m-2k+2)(m-2k+1)(1-r^2)^{m-2k}r^2P_{k-1}(r),
 \end{array}
\end{equation*}
proving the claim. Notice in particular that, since $P_{k-1}$ is even, then $P_{k-1}'$ is odd, meaning that  $r^{-1}P_{k-1}'(r)$ is an even polynomial.

It is now clear that  $\Delta_r^{k}\fr{1}{(1-r^2)^t}\Delta_r^{k}g$ will be a constant, hence we only have to keep track of the higher order monomial.
We will write $R(r)$ to indicate lower order terms that will eventually vanish. Employing formula~\eqref{delta-f}, we obtain
\begin{gather*}
\begin{array}{lll}
\Delta_r^{k}\fr{1}{(1-r^2)^t}\Delta_r^{k}g & = & \Delta_r^{k}\fr{1}{(1-r^2)^t}\Delta_r^{k} ((-1)^m r^{2m}+R(r))\eqskip
& = &(-1)^m \Delta_r^k\fr1{(1-r^2)^t}\left(r^{m+t}\dprod_{j=0}^{k-1}(2m-2j)\dprod_{j=1}^{k}(2m+d-2j)+R(r)\right)\eqskip
& = & (-1)^{m-t} \dprod_{j=0}^{k-1}(2m-2j)\dprod_{j=1}^{k}(2m+d-2j)\Delta_r^k\left(r^{m-t}+R(r)\right)\eqskip
& = & (-1)^{m-t}\fr{(2m)!!(2m+d-2)!!}{(m+t)!!(m+t+d-2)!!}(m-t)!!(m-t+d-2)!!\; .
\end{array}
\end{gather*}

{\bf Case $m-t$ odd.} Write $k=\fr{m-t-1}{2}$. We have to compute
$
(-1)^{m-t}\Delta_r^{k}r^{1-d}\partial_r\fr{r^{d-1}}{(1-r^2)^t}\partial_r\Delta_r^{k}g.
$
We first show that, if $m\ge 2k$, then
$$
\partial_r\Delta_r^{k}g(r) =(1-r^2)^{m-2k-1} r Q_k(r),
$$
where $Q_k$ is an even polynomial of degree $2k$. This follows directly from~\eqref{evenoddcomp}. 

It is now clear that  $\Delta_r^{k}r^{1-d}\partial_r\fr{r^{d-1}}{(1-r^2)^t}\partial_r\Delta_r^{k}g$ will be a constant, hence we only have to keep track of the higher order monomial.
We will write $R(r)$ to indicate lower order terms that will eventually vanish. Using formula~\eqref{delta-f}, we obtain
\begin{gather*}
 \begin{array}{lll}
\Delta_r^{k}r^{1-d}\partial_r\fr{r^{d-1}}{(1-r^2)^t}\partial_r\Delta_r^{k}g & = &\Delta_r^{k}r^{1-d}\partial_r\fr{r^{d-1}}{(1-r^2)^t}\partial_r\Delta_r^{k} ((-1)^m r^{2m}+R(r))\eqskip
& = & (-1)^m \Delta_r^kr^{1-d}\partial_r\fr{r^{d-1}}{(1-r^2)^t}\left((m+t+1)r^{m+t}\dprod_{j=0}^{k-1}(2m-2j)\dprod_{j=1}^{k}(2m+d-2j)+R(r)\right)\eqskip
& = & (-1)^{m-t} (m+t+1) (m-t+d-1)\dprod_{j=0}^{k-1}(2m-2j)\dprod_{j=1}^{k}(2m+d-2j)\Delta_r^k\left(r^{m-t-1}+R(r)\right)\eqskip
& = & (-1)^{m-t}\fr{(2m)!!(2m+d-2)!!}{(m+t+1)!!(m+t+d-1)!!}(m-t-1)!!(m-t+d-1)!!.
\end{array}
\end{gather*}
\end{proof}


\section*{Acknowledgements} DB wishes to thank Luigi Provenzano and Jean Lagac\'e for fruitful discussions on earlier
versions of the manuscript that led to its improvement. DB is a member of the Gruppo Nazionale per l’Analisi Matematica,
la Probabilit\`a e le loro Applicazioni (GNAMPA) of the Istituto Nazionale di Alta Matematica (INdAM). DB was
partially supported by the GNAMPA 2022 project ``Modelli del 4o ordine per la dinamica di strutture ingegneristiche: aspetti analitici
e applicazioni'' and by the GNAMPA 2023 project ``Operatori differenziali e integrali in geometria spettrale''. DB also acknowledges support from the project ``Perturbation problems and asymptotics for elliptic differential equations: variational and potential theoretic methods" funded by the European Union - Next Generation EU and by MUR Progetti di Ricerca di Rilevante Interesse Nazionale (PRIN) Bando 2022 grant 2022SENJZ3.
PF was partially supported by the Funda\c c\~{a}o para a Ci\^{e}ncia e a Tecnologia (Portugal)
through project UIDB/00208/2020.

\subsection*{Data availability} not applicable.

\end{document}